\documentclass[12pt]{article}
\usepackage{amssymb,amsfonts,amsmath,amsthm,float}
\usepackage{fontenc}
\usepackage[utf8]{inputenc}
\usepackage{tkz-euclide, sidecap}
\usepackage{pgf, rotating}
\usepackage{mathrsfs}
\usetikzlibrary{arrows}
 \usepackage{tikz}
\usetkzobj{all}
\usetikzlibrary{patterns, calc}
\usepackage{framed}
\usepackage{pbox}
\usepackage[normalem]{ulem}
\usepackage{color}
\usepackage{enumerate}
\usepackage{comment}

\usepackage{color}
\usepackage{graphicx}

\newcommand{\sign}{\mbox{\rm \small sign}}

\newcommand{\witi}{\widetilde}

\newcommand{\ol}{\overline}
\newcommand{\odin}{\mbox {\bf 1}}

\newcommand{\zz}{{\mathbb Z}}

\newcommand{\cc}{{\mathbb C}}
\newcommand{\nn}{{\mathbb N}}

\newcommand{\rr}{{\mathbb R}}

\newcommand{\cala}{{\mathcal A}}

\newcommand{\cald}{{\mathcal D}}

\newcommand{\calk}{{\mathcal K}}

\newcommand{\veps}{\varepsilon}

\newcommand{\beq}{\begin{eqnarray*}}
\newcommand{\feq}{\end{eqnarray*}}
\newcommand{\beqn}{\begin{eqnarray}}
\newcommand{\feqn}{\end{eqnarray}}
\newcommand{\as}{\mbox{\rm a.\,s.}}

\newtheorem{theorem}{Theorem}
\makeatletter \@addtoreset{theorem}{section}\makeatother

\newtheorem{lemma}[theorem]{Lemma}
\newtheorem{assume}[theorem]{Assumption}

\newtheorem*{theorem*}{Theorem}

\newtheorem{proposition}[theorem]{Proposition}
\newtheorem{corollary}[theorem]{Corollary}
\newtheorem{remark}[theorem]{Remark}

\title{Random walk on the Poincar\'{e} disk induced by \\ a group of M\"{o}bius transformations}
\date{April 20, 2018; Revised: October 21, 2018}
\author{
 Charles McCarthy \thanks{Department of Mathematics, Iowa State University, Ames, IA 50011, USA; e-mail: lasker@iastate.edu}
 \qquad
 Gavin Nop\thanks{Department of Mathematics, Iowa State University, Ames, IA 50011, USA; e-mail: gnnop@iastate.edu}
 \and
Reza~Rastegar\thanks{Occidental Petroleum Corporation, Houston, TX 77046 and Departments of Mathematics and Engineering, University of Tulsa, OK 74104, USA - Adjunct Professor; e-mail:  reza\_rastegar2@oxy.com} \qquad
Alexander~Roitershtein \thanks{Department of Statistics, Texas A\&M University, College Station, TX 77843, USA; e-mail:  \qquad \qquad \qquad alexander@stat.tamu.edu}
}
\usepackage{fancyhdr}
\begin{document}
\maketitle
\begin{abstract}
We consider a discrete-time random  motion, Markov chain on  the  Poincar\'{e}  disk. In the basic variant of the model a particle moves along certain circular arcs within the disk, its location is determined by a composition of random M\"{o}bius transformations. We exploit an isomorphism between the underlying group of M\"{o}bius transformations and $\rr$ to study the random motion through its relation to a one-dimensional random walk. More specifically, we show that key geometric characteristics of the random motion, such as Busemann functions and bipolar coordinates evaluated at its location, and hyperbolic distance from the origin, can be either explicitly computed or approximated in terms of the random walk. We also consider a variant of the model where the motion is not confined to a single arc, but rather the particle switches between arcs of a parabolic pencil of circles at random times.
\end{abstract}

\section{Introduction}
\label{intro}
The aim of this paper is to introduce a certain random motion on the Poincar\'{e} disk and study its basic asymptotic properties. This motion is induced by an iterated random functions (IRF) model, where the functions are M\"{o}bius transformations drawn at random from the family introduced in Section~\ref{prelim}. The majority of our work is thus the study of this specific family of M\"{o}bius transformations and their associated IRF model. The primary feature of our focus is the convergence of iterations. Conditions for the convergence of a sequence of compositions of M\"{o}bius transformations are well-known \cite{m-irf,keym}. A standard condition implies that almost every member in the underlying family of transformations should have a common attractive fixed point. In fact, under mild technical assumption this condition is necessary and sufficient for an IRF model to converge. The underlying family of M\"{o}bius transformations considered in this paper can be characterized by the assumption that all the transformations are one-to-one on the unit disk and have two common fixed points along with a natural maximality property (see Section~\ref{prelim}, in particular Proposition~\ref{observe1}). 
\par 
It turns out that after a proper re-parametrization and using some intrinsic coordinates, various features of our studied IRF can be quantified in terms of a one-dimensional random walk. The relation between the induced motion in the Poincare disk and the random walk on $\rr$ is the key tool we use to study the former. The random motion can be seen as a mixture of basic IRFs in a fashion similar to the way a random walk on $\zz^2$ can be perceived as a mixture of random walks on $\zz.$ Furthermore, an orbit of each of the basic IRF model is contained within an arc in the unit disk connecting two opposite points on the perimeter of the unit disk which are common fixed points of the underlying M\"{o}bius transformations (dotted horizontal lines in Fig.~\ref{fig0} below). The vertical arcs are geodesic lines in the Poincar\'{e} disk model whereas the horizontal are not. In the second part of the paper we consider a genuinely two-dimensional walk on the Poincar\'{e} disk where (similarly to the classical simple random walk on $\zz^2$) at each jump the random walk moves along either dotted or dashed circles and thus preserves one of the two intrinsic coordinates, current direction being decided by a regular coin-tossing procedure. Remark that typically, a discrete-time random walk on a hyperbolic space \cite{hrm, ks, neq} or, more generally, on a Riemanian manifold \cite{jorge,pinsky} is a motion along geodesic lines. 
\\
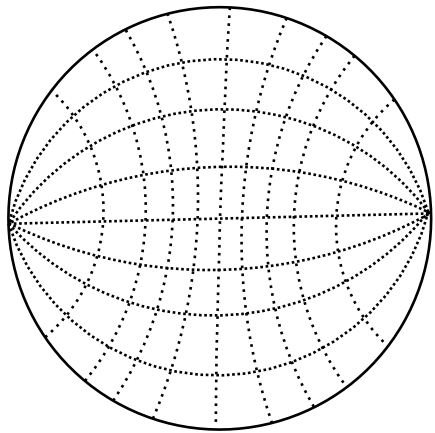
\begin{figure}[!hb]
\begin{center}
\begin{turn}{-30}
\begin{tikzpicture} [scale = 0.7]
\draw [line width=1pt] (-1.12,0.89) circle (4.013277961965755cm);
\draw [line width=1pt,dash pattern=on 1pt off 1pt] (-4.5432974634567795,-1.2046203657680743)-- (2.3075343970440345,2.9776800418359097);
\draw [shift={(2.888431895775272,-5.676025934506512)},line width=1pt,dash pattern=on 1pt off 1pt]  plot[domain=1.637822793434394:2.59995036839282,variable=\t]({1*8.673181021344863*cos(\t r)+0*8.673181021344863*sin(\t r)},{0*8.673181021344863*cos(\t r)+1*8.673181021344863*sin(\t r)});
\draw [shift={(0.3309117026156172,-1.5167803593575142)},line width=1pt,dash pattern=on 1pt off 1pt]  plot[domain=1.166562704597435:3.077636788742166,variable=\t]({1*4.884194801215645*cos(\t r)+0*4.884194801215645*sin(\t r)},{0*4.884194801215645*cos(\t r)+1*4.884194801215645*sin(\t r)});
\draw [shift={(-0.5558387159915673,-0.07467532154864191)},line width=1pt,dash pattern=on 1pt off 1pt]  plot[domain=0.8264715007811323:3.417727992558468,variable=\t]({1*4.144466559847346*cos(\t r)+0*4.144466559847346*sin(\t r)},{0*4.144466559847346*cos(\t r)+1*4.144466559847346*sin(\t r)});
\draw [shift={(-2.9084745154541007,3.751370730804398)},line width=1pt,dash pattern=on 1pt off 1pt]  plot[domain=4.393760447518583:6.133624353000603,variable=\t]({1*5.218667820490378*cos(\t r)+0*5.218667820490378*sin(\t r)},{0*5.218667820490378*cos(\t r)+1*5.218667820490378*sin(\t r)});
\draw [shift={(-1.7541942777881347,1.8741871277835132)},line width=1pt,dash pattern=on 1pt off 1pt]  plot[domain=-2.3068635982392256:0.26787778439923976,variable=\t]({1*4.154293220591969*cos(\t r)+0*4.154293220591969*sin(\t r)},{0*4.154293220591969*cos(\t r)+1*4.154293220591969*sin(\t r)});
\draw [shift={(-6.865410056788786,-2.246793069183498)},line width=1pt,dotted]  plot[domain=-0.16020842873872088:1.1596885902846057,variable=\t]({1*5.13990934404543*cos(\t r)+0*5.13990934404543*sin(\t r)},{0*5.13990934404543*cos(\t r)+1*5.13990934404543*sin(\t r)});
\draw [shift={(144.40497742708402,78.57719125607755)},line width=1pt,dotted]  plot[domain=3.6076161548313506:3.6562740347475775,variable=\t]({1*164.96308296225456*cos(\t r)+0*164.96308296225456*sin(\t r)},{0*164.96308296225456*cos(\t r)+1*164.96308296225456*sin(\t r)});
\draw [shift={(-5.835374337263382,-1.6693357214578461)},line width=1pt,dotted]  plot[domain=-0.3413313931989901:1.3358730295534569,variable=\t]({1*3.1561208595190156*cos(\t r)+0*3.1561208595190156*sin(\t r)},{0*3.1561208595190156*cos(\t r)+1*3.1561208595190156*sin(\t r)});
\draw [shift={(-11.013438819784929,-4.354498954772841)},line width=1pt,dotted]  plot[domain=0.12209367085772392:0.8527775761337828,variable=\t]({1*10.778138359322298*cos(\t r)+0*10.778138359322298*sin(\t r)},{0*10.778138359322298*cos(\t r)+1*10.778138359322298*sin(\t r)});
\draw [shift={(-7.75973732824131,-2.656969766613228)},line width=1pt,dotted]  plot[domain=-0.07031859914041849:1.0515874226105568,variable=\t]({1*6.636121301158943*cos(\t r)+0*6.636121301158943*sin(\t r)},{0*6.636121301158943*cos(\t r)+1*6.636121301158943*sin(\t r)});
\draw [shift={(-5.128431895775272,7.456025934506511)},line width=1pt,dash pattern=on 1pt off 1pt]  plot[domain=4.779415447024187:5.741543021982613,variable=\t]({1*8.673181021344863*cos(\t r)+0*8.673181021344863*sin(\t r)},{0*8.673181021344863*cos(\t r)+1*8.673181021344863*sin(\t r)});
\draw [shift={(3.5953743372633817,3.449335721457846)},line width=1pt,dotted]  plot[domain=2.800261260390803:4.47746568314325,variable=\t]({1*3.156120859519015*cos(\t r)+0*3.156120859519015*sin(\t r)},{0*3.156120859519015*cos(\t r)+1*3.156120859519015*sin(\t r)});
\draw [shift={(8.773438819784928,6.13449895477284)},line width=1pt,dotted]  plot[domain=3.2636863244475167:3.9943702297235757,variable=\t]({1*10.778138359322297*cos(\t r)+0*10.778138359322297*sin(\t r)},{0*10.778138359322297*cos(\t r)+1*10.778138359322297*sin(\t r)});
\draw [shift={(5.519737328241309,4.4369697666132275)},line width=1pt,dotted]  plot[domain=3.0712740544493746:4.19318007620035,variable=\t]({1*6.636121301158943*cos(\t r)+0*6.636121301158943*sin(\t r)},{0*6.636121301158943*cos(\t r)+1*6.636121301158943*sin(\t r)});
\draw [shift={(4.625410056788786,4.026793069183498)},line width=1pt,dotted]  plot[domain=2.9813842248510722:4.301281243874398,variable=\t]({1*5.13990934404543*cos(\t r)+0*5.13990934404543*sin(\t r)},{0*5.13990934404543*cos(\t r)+1*5.13990934404543*sin(\t r)});
\end{tikzpicture}
\end{turn}
\caption{
Two orthogonal pencils of coaxial circles. Dashed circles are centered on the $y$ axis and go through the points $-1$ and $1.$ Dotted circles are centered on the $x$ axis and intersect dashed circles at the right angle. }
\label{fig0}
\end{center}
\end{figure}
\par
 The paper is structured as follows. In Section~\ref{prelim} we introduce a family (an abelian subgroup) of M\"{obius} transformations that governs the ``horizontal" portion of the random motion. In Sections~\ref{loco} and~\ref{converge} we consider the basic IRF that represents the horizontal part of the random motion in our model. The functions in the IRF model are M\"{o}bius transformations drawn at random from the family introduced in Section~\ref{prelim}. To study the IRF we introduce in Section~\ref{loco} so-called bipolar coordinates whose isoparametric curves are represented by two orthogonal sets (pencils) of circular arcs (see Fig.~\ref{fig0}). In Section~\ref{loco} we compute certain geometric characteristics of the auxiliary IRF process, and in Section~\ref{converge} we discuss the rate of its attraction to the disk boundary in various regimes. In Section~\ref{apencils} we apply these results to study a genuinely two-dimensional motion on $\cald$ that is in some natural fashion induced by the above auxiliary process. These are the main results of this paper and are concerned with the asymptotic behavior of the random motion. Finally, in Section~\ref{simul} we present some numerical simulations for the motion introduced in Section~\ref{apencils}.

\section{An abelian group of M\"{o}bius transformations}
\label{prelim}
In this section we introduce a family of M\"{o}bius transformations that govern the ``horizontal" (i.\,e., along dashed arcs in Fig.~\ref{fig0} above) motion in our model.
\par
Let $\cald$ be open unit disk with center at $0$ in the complex plane $\cc$. Recall that the Poincar\'{e} distance $d_{\mathfrak p}(z,w)$ between two points $z,w\in\cald$ is given by
\beqn
\label{poincare}
d_{\mathfrak p}(z,w)=\log\left(\frac{1+\rho_{z,w}}{1-\rho_{z,w}}\right)
=\mbox{\rm \small arctanh}\,\rho_{z,w},
\qquad \mbox{\rm where}~\rho_{z,w}:=\Bigr|\frac{z-w}{1-\bar zw}\Bigr|.
\feqn
The choice of the constant in front of the logarithm is a matter of convention, it is often set to be $\frac{1}{2}.$ The Poincar\'{e} disk $(\cald, d_{\mathfrak p})$ is a canonical model of the hyperbolic space. Let $\ol \cald$ be the closure of $\cald$ in $\cc$ and $\partial \cald $ be its boundary, a unit circle. Let $\alpha\in\partial \cald$ be given, it will remain fixed throughout the paper. 
\par 
We will be concerned with M\"{o}bius transformations $q_{x,\alpha}:\ol \cald \to \ol \cald$
defined for $\alpha \in \partial \cald$ and real $x\in (-1,1)$ by
\beqn
\label{qxalpha}
q_{x,\alpha}(z)=\frac{z+x\alpha}{1+x\ol \alpha z}, \qquad z\in\ol \cald.
\feqn
In the language of \cite{gyro1,gyro}, the kind of M\"{o}bius transformations that appears in \eqref{qxalpha} (namely, $z\oplus w:=\frac{z+w}{1+\ol wz}$ with $w=x\alpha$) is called a gyrotranslation. This particular class of M\"{obius} transformations plays a fundamental role in hyperbolic geometry and its applications (for instance, notice that $\rho_{z,w}$ in \eqref{poincare} is $|z\oplus (-w)|$).
\par
Remark that for any $\alpha \in \partial \cald$ and $x\in (-1,0) \cup (0,1),$
\begin{enumerate}[(i)]
\item $q_{x,\alpha}$ is a holomorphic automorphism (one-to-one and onto map) of $\cald.$
\item $q_{x,a}$ stabilizes the boundary $\partial \cald.$ That is, $q_{x,\alpha}(\partial \cald)=\partial \cald.$
\item $q_{x,\alpha}$ has two fixed points, $-\alpha$ and $\alpha.$ That is, $q_{x,\alpha}(-\alpha)=-\alpha$ and $q_{x,\alpha}(\alpha)=\alpha.$
\end{enumerate}
In fact, any transformation of $\cald$ that has property (i) is in the form $\beta q_{x,\alpha}$ for some $\alpha,\beta\in\partial\cald,$ and hence posses property (ii) as well. Let $Q_\alpha:=\{q_{x,\alpha}:x\in(-1,1)\}.$ Note that $Q_\alpha$ is closed under finite compositions of its elements. More precisely, for any $x_1,x_2\in(-1,1),$ there is a unique $x\in (-1,1)$ such that
\beqn
\label{fcomp}
q_{x_1,\alpha}\circ q_{x_2,\alpha}=
q_{x,\alpha},\qquad \mbox{\rm where}~\frac{1+x}{1-x}=\frac{1+x_1}{1-x_1}\cdot \frac{1+x_2}{1-x_2}.
\feqn
Our results in the next sections rely on the asymptotic properties of iterated compositions of random elements of $Q_\alpha.$ Property (iii) which guarantees that all transformations in the family $Q_\alpha$ have a common fixed point is the main reason why we exclude the composition with a rotation from our consideration.  It is well-known that a sequence of iterated M\"{o}bius transformation necessarily diverges unless all the transformations have a common fixed point w.p.1 (with probability one) \cite{m-irf,keym}. The feature of $Q_\alpha$'s that originally spiked our interest in these collections of M\"{o}bius transformations can be described as follows:
\begin{proposition}
\label{observe1}
$Q_\alpha$ for $\alpha\in\partial\cald$ is a maximal family of transformations on $\ol \cald$ such that its elements are holomorphic bijections on $\cald,$ all having two common fixed points. Conversely, any such maximal family is either $Q_\alpha$ for some $\alpha \in\partial\cald$ or consists of exactly one element. 
\end{proposition}
\begin{proof}
Any holomorphic bijections on $\cald$ has the form $f(z)=\beta q_{x,\alpha}=\beta \frac{z+x\alpha}{1+x\ol \alpha z}$ for some $\alpha,\beta \in \partial \cald$ and $x\in (-1,1).$
Our goal is to show that either $\beta=1$ or there is no other holomorphic bijections on $\cald$ having the same two fixed point as $f.$  Let $z=\alpha \omega$ be a non-zero fixed point of $f.$  Then
\beqn
\label{om}
\omega =\beta \frac{\omega+x}{1+x\ol \omega}.
\feqn 
Since $|\beta|=1$ this implies that $\bigl|\omega +x|\omega|^2\bigr|=|\omega +x|.$ It is easy to see that this identity is only possible if either $|\omega|=1$ or $x=0$ or $x=-\frac{2Re(\omega)}{1+|\omega|^2}.$ If $x=0$ then $\beta=1$ and hence $f\in Q_\alpha.$ If $|\omega|=1$ then \eqref{om} yields $\omega+x=\beta(\omega+x)$ and thus $\beta=1$ and $f\in Q_\alpha.$ Finally, if $x=-\frac{2Re(\omega)}{1+|\omega|^2}$ and $|\omega|\neq 1$ then $x$ and due to \eqref{om} also $\beta$ are uniquely determined by $\omega,$ implying that $f$ is a unique holomorphic bijections on $\cald$ having $z=\alpha \omega$ as its fixed point.  
\end{proof}
In view of \eqref{fcomp}, the family of transformations $Q_\alpha$ is more conveniently parametrized using the parameter
\beqn
\label{poi}
\gamma(x)=\log \Bigl(\frac{1+x}{1-x}\Bigr)=\sign(x)\cdot d_{\mathfrak p}(0,x\alpha)
\feqn
rather than $x$ (see \eqref{grap} below). Note that $\gamma:(-1,1)\to\rr$ is a one-to-one transformation. In addition to the representation as the right-most expression in \eqref{poi}, another illuminating interpretation of $\gamma(x)$ is that
\beqn
\label{der}
e^{\gamma(x)}=\frac{1+x}{1-x}=\frac{d}{dz}q_{x,\alpha}(-\alpha)=\Bigl(\frac{d}{dz}q_{x,\alpha}(\alpha)\Bigr)^{-1}.
\feqn
In what follows, for a given $\alpha\in\partial\cald,$  we denote
\beqn
\label{ggamma}
g_\gamma:=q_{x,\alpha}\qquad \mbox{\rm if}\qquad \gamma=\log \Bigl(\frac{1+x}{1-x}\Bigr).
\feqn
In this notation \eqref{fcomp} reads
\beqn
\label{grap}
g_{\gamma_1}\circ g_{\gamma_2}=g_{\gamma_2}\circ g_{\gamma_1}=g_{\gamma_1+\gamma_2}.
\feqn
Thus
\beqn
\label{ga}
G_\alpha:=\{g_\gamma:\gamma\in\rr\}
\feqn
form an abelian group of transformations on $\ol \cald.$
\par
Let $x=(x_n)_{n\in\nn}$ be a stationary and ergodic sequence of random variables valued in the interval $(-1,1).$ Specific assumptions on the distribution of $x$ will be made in Section~\ref{converge}.
Let
\beq
\rho_n=\frac{1+x_n}{1-x_n}\qquad \mbox{\rm and}\qquad \gamma_n=\log \rho_n.
\feq
Consider a (generalized) random walk on $\rr$ defined by the partial sums
of the sequence $\gamma_n:$
\beqn
\label{rrw}
\omega_0:=0\qquad \mbox{\rm and}\qquad \omega_n:=\sum_{i=1}^n \gamma_i,\quad  n\in\nn.
\feqn
For a comprehensive treatment of regular random walks on $\rr$ (partial sums of i.\,i.\,d. random variables) see, for instance, monographs \cite{borovs, budapest}. For recent results and literature survey on partial sums of dependent real-valued variables see, for instance, \cite{clt4,mclt,clt5}.
\par
The random walk $\omega_n$ induces a random walk on the elements of the group $G_\alpha$:
\beq
U_n:=g_{\omega_n},\qquad n\in\zz_+.
\feq
Here and thereafter, $\zz_+$ stands for the set of non-negative integers $\nn\cup\{0\}.$
\par
It follows from \eqref{grap} that for any permutation $\sigma:\{1,2,\ldots,n\}\to \{1,2,\ldots,n\},$
\beqn
\label{permute}
U_n=g_{\gamma_{\sigma(1)}}\circ g_{\gamma_{\sigma(2)}}\circ \ldots \circ g_{\gamma_{\sigma(n)}},
\feqn
and
\beqn
\label{r1}
U_n(z)=\frac{z+y_n\alpha}{1+y_n\bar \alpha z},
\feqn
with $y_n\in(-1,1)$ such that
\beqn
\label{yrho}
\frac{1+y_n}{1-y_n}=\prod_{k=1}^n \rho_k=e^{\omega_n}.
\feqn
In view of \eqref{der} and \eqref{permute}, the identity \eqref{yrho} can be thought as a direct implication of the chain rule formula $U_n'(\alpha)=g_{\gamma_n}'(\alpha)\cdot U_{n-1}'(\alpha).$
\par
We observe in Lemma~\ref{cinv} below that the entire orbit $\{U_n(z):n\in\nn\}$ lies on the circular arc going through the poles $-\alpha,\alpha,$ and  the initial position $z\in\ol \cald.$ 

\section{Locus of the orbits $\{U_n(z)\}_{n\in\nn}$ in the hyperbolic disk}
\label{loco}
The goal of this section is to explicitly relate certain geometric characteristics of $U_n(z)$ to the random walk $\omega_n.$ The relation can be subsequently exploited in order to deduce the asymptotic behavior of $U_n(z)$ from standard limit theorems for $\omega_n.$
\par
The main results of this section are Lemma~\ref{cinv} asserting that the orbit $\bigl(U_n(z)\bigr)_{n\in\zz_+}$ belongs to a circle which is invariant under the transformations of group $G_\alpha,$ Lemma~\ref{pbuz} that can serve, for instance, to calculate the rate of convergence of a transient process $U_n$ to the boundary (see Theorem~\ref{escape} in Section~\ref{converge}), and Lemma~\ref{bipo} that establishes the location of $U_n(z)$ in terms of certain, intrinsic to the problem, bipolar coordinates (for an example of application see Proposition~\ref{plln} in Section~\ref{converge}).
\par
Let $\ol l_\alpha$ and $l_\alpha$ denote, respectively, the diameter of $\cald$ connecting $\alpha$ and $-\alpha$ and its interior. That is,
\beqn
\label{la}
l_\alpha=\{z\in\cald: z=x\alpha~\mbox{\rm for some}~x\in(-1,1)\}\qquad \mbox{\rm and}\qquad \ol l_\alpha=l_\alpha \cup \{-\alpha,\alpha\}.
\feqn
Note that $\ol l_\alpha$ is invariant under $U_n$ for all $n\in\zz.$  In fact,
\beqn
\label{lal}
U_n(x\alpha)=\alpha \frac{x+y_n}{1+xy_n},\qquad x\in [-1,1].
\feqn
In the general case we have the following lemma:
\begin{lemma}
\label{cinv}
For an arbitrary $z\in\ol \cald\backslash\{-\alpha,\alpha\},$ the orbit $\{U_n(z)\}_{n\in\nn}$ lies entirely on the circle $H_z$ through the three points $\alpha,-\alpha,$ and $z.$
\end{lemma}
\begin{remark}
\label{d}
If $z\in l_\alpha,$ then the circle $H_z$ is degenerate according to \eqref{lal}. Namely, it is the straight line that contains the segment $l_\alpha.$
\end{remark}
\begin{proof}[Proof of Lemma~\ref{cinv}]
For $z\in l_\alpha$ the result follows directly from \eqref{lal}. For $z\in\ol \cald\backslash \ol l_\alpha,$ letting $u_n=y_n\alpha,$ we get $U_n(z)=\alpha^2\frac{z+u_n}{\alpha^2+zu_n}.$ Thus for a fixed $z\in\ol\cald\backslash \ol l_\alpha,$ the orbit
$\{U_n(z)\}_{n\in\zz_+}$ lies entirely within the image of $l_\alpha$ under the M\"{o}bius transformation
\beqn
\label{tea}
T_z(u):=\alpha^2\frac{z+u}{\alpha^2+zu}.
\feqn
The image of a line under a M\"{o}bius transformation is either a line or a circle. The result then follows from the observation that $T_z(-\alpha)=-\alpha,$ $T_z(\alpha)=\alpha,$ and $T_z(0)=z.$
\end{proof}
\begin{remark}
\label{circles}
\item[(i)] The proof of Lemma~\ref{cinv} which is given above suggests the following parametrization of the arc $H_z\cap \cald.$ Recall that the arc is the image of the diameter $l_\alpha$ under the transformation $T_z.$ Thus it is natural to parametrize it by associating a point $w\in H_z\cap \cald$ with the real parameter $\ol \alpha T^{-1}_z(w),$ or perhaps with $\gamma\bigl( \ol \alpha T^{-1}_z(w)\bigr),$ where $\gamma:(-1,1)\to\rr$ is defined in \eqref{poi}. The former variation is essentially (up to a constant) the chord-length parametrization of the arc $H_z\cap \cald $ in the language of \cite{bp1}. For the sake of technical convenience, we choose the latter variant. More specifically, define $\hat \tau_z : \rr:\to H_z\cap \cald$ by
\beqn
\label{bpp}
\hat \tau_z(\delta)=T_z\bigl(\gamma^{-1}(\delta)\cdot \alpha\bigr).
\feqn
Notice that due to the isometry property of M\"{o}bius transformations, for any \beq \delta_1=\gamma(x_1)=\log\Bigl(\frac{1+x_1}{1-x_1}\Bigr)\quad \mbox{\rm and} \quad \delta_2 =\gamma(x_2)=\log\Bigl(\frac{1+x_2}{1-x_2}\Bigr),\qquad x_1,x_2\in (-1,1),
\feq
we have
\beq
d_{\mathfrak p}\bigl(\hat \tau_z(\delta_1),\hat \tau_z(\delta_2)\bigr)=
d_{\mathfrak p}(x_1\alpha,x_2\alpha)=d_{\mathfrak p}(x_1,x_2)=|\gamma(x_1)-\gamma(x_2)|.
\feq
A slight modification of the parametrization $\hat \tau_z$ introduced in \eqref{bpp} plays a central role in our study (cf. \eqref{bipolar}, Remark~\ref{biporem} and Lemma~\ref{bipo} below). Note that by virtue of \eqref{r1} and \eqref{yrho}, the definition \eqref{bpp} implies that
\beqn
\label{bipovar}
U_n(z)=\hat \tau_z(\omega_n).
\feqn
\item[(ii)] An alternative short proof of the lemma follows from the parametric definition of the circle $H_z$ given in \eqref{circle} below and \eqref{curious} which identifies the (hyperbolic, e. g. having two distinct fixed points) M\"{o}bius transformation $U_z(z)$ in terms of its fixed points and derivatives at those points (see, for instance, \cite[p.~22]{localv}).
\end{remark}
For a given initial point $z\in\ol \cald\backslash\{-\alpha,\alpha\},$ an equation of the circle in Lemma~\ref{cinv} can be written as
\beqn
\label{circle}
H_z=\Bigl\{u\in\cc: \frac{u+\alpha}{u-\alpha}\cdot \frac{z-\alpha}{z+\alpha}\in \rr\Bigr\}.
\feqn
Note that for $z\in l_\alpha,$ the circle $H_z$ is degenerate, namely $H_z$ is the straight line which contains the segment $l_\alpha.$  A direct computation exploiting \eqref{r1} shows that for $u=U_n(z)$ in \eqref{circle} we have
\beqn
\label{curious}
\frac{U_n(z)+\alpha}{U_n(z)-\alpha}=\frac{z+\alpha}{z-\alpha} \cdot \frac{1+y_n}{1-y_n}=\frac{z+\alpha}{z-\alpha} \cdot e^{\omega_n},\qquad z\in\ol \cald\backslash\{-\alpha,\alpha\}.
\feqn
Writing $\frac{1+y_n}{1-y_n}$ as $\frac{\alpha+y_n\alpha }{\alpha-y_n\alpha},$ the last formula can be seen as a direct implication of the following instance of the cross-ratio property of the M\"{o}bius transformation $U_n:$
\beq
\frac{U_n(z)-U_n(-\alpha)}{U_n(z)-U_n(\alpha)}\cdot \frac{U_n(0)-U_n(\alpha)}{U_n(0)-U_n(-\alpha)}=\frac{z-(-\alpha)}{z-\alpha}\cdot \frac{0-\alpha}{0-(-\alpha)}.
\feq
It is not hard to verify that the center of the circle $H_z$ is at the point $-i\alpha c$ and its radius is $\sqrt{1+c^2},$ where $c$ is the unique solution to $|z+ic\alpha|^2=1+c^2,$ that is
$c=\frac{1-|z|^2}{2i(\ol z\alpha-z\ol \alpha) }.$ These explicit formulas are not used anywhere in the sequel and are given here for the sake of completeness only.
\par
We next consider \emph{bipolar coordinates} \cite{bp, bp3, bp1} associated with the pair $(\alpha,-\alpha).$ For $z\in\cc$ we let
\beqn
\label{bipolar}
\tau(z)=\log \frac{|z+\alpha|}{|z-\alpha|} \qquad \mbox{\rm and} \qquad \sigma(z)=
\sign\bigl(\mbox{\rm \small Im}(\ol \alpha z)\bigr)\cdot \angle (-\alpha, z,\alpha),
\feqn
where $\angle(A,B,C)\in (0,\pi]$ is the angle $B$ in the triangle formed by the vertexes $A,B,$ and $C$ in the complex plane. Thus $\tau(z)\in (0,\infty),$ $\sigma(z)=\pi$ if $z\in l_\alpha,$ $\sigma(z)>0$ if the orientation of the circular arc going from $\alpha$ to $-\alpha$ through $z$ is anticlockwise, $\sigma(z)<0$ otherwise,  and we convene that $\sigma\in\bigl(\frac{\pi}{2},\pi\bigr]\bigcup \bigl(-\pi,-\frac{\pi}{2}\bigr)$ for $z\in\cald$ (see Fig.~\ref{fig1} below).
\par
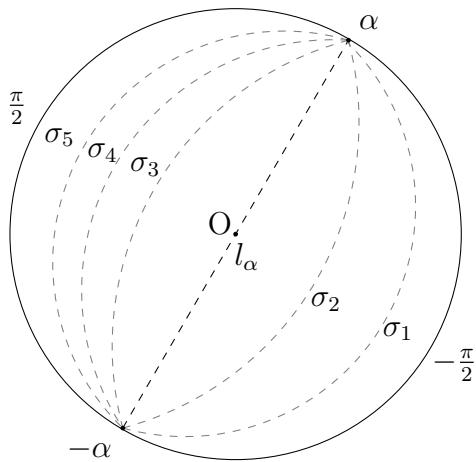
\begin{figure}[!hb]
\begin{center}
\usetikzlibrary{calc,through}
\begin{tikzpicture}[scale=3.0]
\draw (5, 0) circle (1);
\draw[dashed] (4.5,-0.86) node[below left] {$-\alpha$} --
(5.5,0.86) node[above right] {$\alpha$};
\tkzDefPoint(5,0){O}
\tkzDefPoint(5.86,-0.5){C}
\tkzDefPoint(5.43,-0.25){D}
\tkzDefPoint(4.14,0.5){E}
\tkzDefPoint(5.215,-0.125){F}
\tkzDefPoint(4.775,0.125){G}
  \tkzDefPoint(4.5,-0.86){B}
  \tkzDefPoint(5.5,0.86){A}
  \tkzDrawArc[dashed](C,A)(B)
  \tkzDrawArc[dashed](D,A)(B)
  \tkzDrawArc[dashed](E,B)(A)
   \tkzDrawArc[dashed](G,B)(A)
  \tkzDrawArc[dashed](F,A)(B)
  \node at (5.05,-0.1) {$l_\alpha$};
  \node at (5.72,-0.43) {$\sigma_1$};
  \node at (5.40,-0.3) {$\sigma_2$};
  \node at (4.60,0.3) {$\sigma_3$};
  \node at (4.41,0.36) {$\sigma_4$};
  \node at (4.22,0.43) {$\sigma_5$};
  \node at (4.03,0.56) {$\frac{\pi}{2}$};
  \node at (5.97,-0.58) {$-\frac{\pi}{2}$};
  \fill (5,0) circle (0.3pt);
  \node at (4.93,0.05) {O};
  \fill (4.5,-0.86) circle (0.3pt);
  \fill (5.5,0.86) circle (0.3pt);
\end{tikzpicture}
\caption{The circles $H_z$ are isoparametric curves for the bipolar coordinate $\sigma,$ that is the value of $\sigma$ is constant on each circle. For $\sigma$'s associated with the dashed circles in the picture, we have $-\pi<\sigma_2<\sigma_1<-\frac{\pi}{2}$ and
$\frac{\pi}{2}<\sigma_5<\sigma_4<\sigma_3<\pi.$ Furthermore, $\sigma(z)=\pm\frac{\pi}{2}$ for $z\in\partial\cald$ and $\sigma(z)=\pi$ for $z\in\ol l_\alpha.$ The location of a point $z$ on a dashed circle can be determined using the second coordinate $\tau,$ which is the logarithm of the ratio of the Euclidean distances from $z$ to $-\alpha$ and from $z$ to $\alpha.$ It can be shown that $\tau$ is constant on arcs orthogonal to the dashed ones (see Fig.~\ref{fig0} above).}
\label{fig1}
\end{center}
\end{figure}
\par
It can be shown that \cite{bp3}
\beqn
\label{bpz}
z=\alpha i \cot \Bigl(\frac{\sigma(z)+i\tau(z)}{2}\Bigr).
\feqn
In a somewhat more explicit form, if $z=x+iy,$ then assuming for simplicity $\alpha=1,$ one obtains  \cite{bp, bp3, bp1}
\beqn
\label{bpxy}
x=\frac{\sinh \tau(z)}{\cosh \tau(z)-\cos \sigma(z)}\qquad \mbox{\rm and} \qquad
y=\frac{\sin \sigma(z)}{\cosh \tau(z)-\cos \sigma(z)}.
\feqn
The bipolar coordinates are commonly used to separate variables in equations of mathematical physics when the Laplace operator is the dominated term (the idea goes back to \cite{bpa}, see, for instance, \cite{bp6,bp4,bp3,bp5}).
\par
Lemma~\ref{cinv} and \eqref{curious} together imply the following:
\begin{lemma}
\label{bipo}
For any $n\in\zz_+$ and $z\in\cald\backslash\{-\alpha,\alpha\}$ we have
\beq
\tau\bigl(U_n(z)\bigr)=\tau(z) + \omega_n \qquad \mbox{\rm and}\qquad \sigma\bigl(U_n(z)\bigr) = \sigma(z).
\feq
\end{lemma}
Thus the displacement of the $\tau$-coordinate $\tau\bigl(U_n(z)\bigr)$ relatively to the initial value $\tau(z)$
is equivalent to a partial sum of suitable i.\,i.\,d. variables.
\begin{remark}
\label{biporem}
A comparison Lemma~\ref{bipo} with \eqref{bipovar} indicates that the function $\tau(u)$ is a shifted version of the inverse of $\hat \tau_z(u)$ introduced in \eqref{bpp} for any $z\in\cald.$ Indeed, it is not hard to verify that in general
\beq
\tau\bigl(g_{\gamma}(z)\bigr)=\tau(z)+\gamma=\tau(z)+\hat \tau_z^{-1}\bigl(g_{\gamma}(z)\bigr), \qquad\forall~\gamma\in\rr,z\in\cald,
\feq
where $g_{\gamma}$ is the M\"{o}bius transformation defined in \eqref{ggamma}.
\end{remark}
The following result (cf. formula (8) in \cite{circ}) is equivalent to the particular case of Lemma~\ref{bipo} when $|z|=1.$
\begin{corollary}
\label{angular}
For $n\in\zz_+$ and $z\in\partial\cald\backslash\{-\alpha,\alpha\},$ let
\beq
\psi_n^-(z):=\angle\bigl(U_n(z),0,\alpha\bigr)
\qquad
\mbox{\rm and}
\qquad
\psi_n^+(z):=\angle\bigl(U_n(z),0,-\alpha\bigr).
\feq
Then, for all $n\in\zz_+$ and $z\in\partial \cald\backslash\{-\alpha,\alpha\},$
\beq
\tan \Bigl(\frac{\psi_n^+(z)}{2}\Bigr)=e^{\omega_n}\cdot \tan \Bigl(\frac{\psi_0^+(z)}{2}\Bigr),
\feq
and, correspondingly,
\beq
\tan \Bigl(\frac{\psi_n^-(z)}{2}\Bigr)=e^{-\omega_n}\cdot \tan \Bigl(\frac{\psi_0^-(z)}{2}\Bigr).
\feq
\end{corollary}
\begin{proof}{Proof of Corollary~\ref{angular}} Note that for $z\in\partial\cald\backslash l_\alpha,$
\beqn
\label{halfa}
\frac{\psi_n^-}{2}=\angle\bigl(U_n(z),-\alpha,\alpha\bigr)\qquad \mbox{\rm and}  \qquad \frac{\psi_n^+}{2}=\angle\bigl(U_n(z),\alpha,-\alpha\bigr).
\feqn
Thus for $z\in \partial \cald,$ we have
\beq
\tau\bigl(U_n(z)\bigr)=\log \Bigl(\tan \frac{\psi_n^+}{2}\Bigr),
\feq
and hence the result in Corollary~\ref{angular} is a direct implication of Lemma~\ref{bipo}.
\end{proof}
Recall that in the Poincar\'{e} disc model, Busemann function corresponding to $\xi\in\partial\cald$ is given by \cite{bbuzu,ag}
\beqn
\label{buz}
B_\xi(z)=-\log\Bigl(\frac{1-|z|^2}{|\xi-z|^2}\Bigr),\qquad z\in \cald.
\feqn
Note that $B_\xi(0)=0,$ $\lim_{z\to\xi} B_\xi(z)=-\infty,$ and $\lim_{z\to\zeta} B_\xi(z)=+\infty$ for any $\zeta\in\cald$ such that $\zeta\neq \xi.$ Intuitively, a Busemann function $B_\xi$ is the distance function on $\cald$ from a point $\xi\in\partial \cald$  ``at infinity". A rigorous definition utilizing this heuristics  and the disc boundary $\partial \cald$ as ``the infinity" is as follows:
\beq
B_\xi(z)=\lim_{x\to1^-} \bigl[d_{\mathfrak p}(z,x\xi)-d_{\mathfrak p}(0,x\xi)],\qquad \xi\in\partial \cald,\,z\in\cald.
\feq
The function within the brackets in \eqref{buz} is the Poisson kernel of the unit disk, the density of a harmonic measure on the disk's boundary \cite{doob1,doob}. The interpretation of $\frac{1}{2\pi}\frac{1-|z|^2}{|e^{i\theta}-z|^2}\,d\theta$ as the density of the standard (Euclidean) planar Brownian motion starting at $z\in\cald$ and evaluated at the first hitting time of $\partial \cald$ provides a link between the Busemann functions and the theory of Brownian motion. We remark that in the context of statistical applications with circular data, the distribution on $\partial \cald$ with this density is often referred to as a circular or wrapped Cauchy distribution \cite{circ}. The curious link between Busemann functions, the probabilistic interpretation of the Poisson kernels, and canonical iterative models of circular data in statistics is not used anywhere in the sequel and is brought here for the sake of completeness only.
\par
From \eqref{yrho} we obtain the following expression for $y_n:$
\beqn
\label{yg}
y_n=\frac{e^{\omega_n}-1}{e^{\omega_n}+1}.
\feqn
Substituting \eqref{r1} and \eqref{yg} into \eqref{buz} yields the following:
\begin{proposition}
\label{pbuz}
For all $z\in \cald$ and $n\in\zz_+,$
\beq
B_\alpha\bigl(U_n(z)\bigr)=-\omega_n+B_\alpha(z)\qquad \mbox{\rm and}\qquad
B_{-\alpha}\bigl(U_n(z)\bigr)=\omega_n+B_{-\alpha}(z).
\feq
\end{proposition}
This observation allows us to immediately translate any result for the random walk $\omega_n$ into a counterpart statement about the behavior of Buzemann functions $B_{\pm\alpha}\bigl(U_n(z)\bigr).$ For an example of application of Proposition~\ref{pbuz} see Theorem~\ref{escape} below.
\begin{remark}
\label{buza}
The result in Proposition~\ref{pbuz} can be also seen as a direct implication of the formula \eqref{curious} because of the Poisson kernel identity $\frac{1-|z|^2}{|\xi-z|^2}=\mbox{\rm \small Re}\bigl(\frac{\xi+z}{\xi-z}\bigr),$ $\xi\in\partial \cald$ and $z\in\cald.$
\end{remark}
The following lemma complements the result in Lemma~\ref{cinv} by highlighting the choice of the direction of the motion along the circle $H_z.$
\begin{lemma}
\label{mon}
For all $n\in\nn,$ $z\in\cald\backslash \{-\alpha,\alpha\}$ and $\veps\in\{-1,1\},$ the following statements are equivalent:
\begin{enumerate}[(i)]
\item $|U_n(z)-\veps \alpha|<|U_{n-1}(z)-\veps \alpha|.$
\item $\veps x_n>0.$
\item $\veps \gamma_n>0.$
\item $B_{\veps\alpha}\bigl(U_n(z)\bigr)<
B_{\veps\alpha}\bigl(U_{n-1}(z)\bigr).$
\end{enumerate}
\end{lemma}
\begin{proof}[Proof of Lemma~\ref{mon}]
The equivalence of (ii) and (iii) is trivial. The equivalence of (iii) and (iv) follows from Proposition~\ref{pbuz}. It remains to show the equivalence of (i) and (ii).
To this end observe that the orbit $\{U_n(z)\}_{n\in\zz_+}$ is the image of $\{y_n\alpha\}_{n\in\zz_+}$ under a conformal transformation $T_z,$ which is defined in \eqref{tea}. By virtue of Lemma~\ref{cinv}, the claim is a consequence of the fact that $T_z$ is preserving the orientation of curves,
\end{proof}
An explicit formula for the hyperbolic distance $d_{\mathfrak p}\bigl(0,U_n(z)\bigr)$ in terms of the initial bipolar coordinates $\bigl(\sigma(z),\tau(z)\bigr)$ can be in principle obtained from \eqref{r1} and \eqref{bpxy}. We conclude this subsection with the following approximation result.
\begin{lemma}
\label{zero}
For all $z\in\cald$ and $n\in\nn,$ we have $\bigl|d_{\mathfrak p}\bigl(0,U_n(z)\bigr)-|\omega_n|\bigr|\leq d_{\mathfrak p}(0,z).$
\end{lemma}
\begin{proof}[Proof of Lemma~\ref{zero}]
First, observe that since $0=U_n(-y_n\alpha),$
\beq
d_{\mathfrak p}\bigl(U_n(0),0\bigr)=d_{\mathfrak p}(0,-y_n\alpha)=\log\Bigl(\frac{1+|y_n|}{1-|y_n|}\Bigr)=\Bigl|\log\Bigl(\frac{1+y_n}{1-y_n}\Bigr)\Bigr|=|\omega_n|,
\feq
where the first identity is implied by the isometry property of  M\"{o}bius transformation $U_n.$ Next, the identity $d_{\mathfrak p}\bigl(U_n(0),0\bigr)=|\omega_n|$ can be utilized to deduce the asymptotic behavior of $d_{\mathfrak p}\bigl(U_n(z),0\bigr)$ for an arbitrary
$z\in\cald$ through the following triangle inequalities:
\beqn
\label{i1}
d_{\mathfrak p}\bigl(U_n(z),0\bigr)\leq d_{\mathfrak p}\bigl(U_n(z),U_n(0)\bigr)+d_{\mathfrak p}\bigl(U_n(0),0\bigr)=
d_{\mathfrak p}(z,0)+d_{\mathfrak p}\bigl(U_n(0),0\bigr)
\feqn
and
\beqn
\label{i3}
d_{\mathfrak p}\bigl(U_n(z),0\bigr)\geq d_{\mathfrak p}\bigl(U_n(0),0\bigr)-d_{\mathfrak p}\bigl(U_n(0),U_n(z)\bigr)=
d_{\mathfrak p}\bigl(U_n(0),0\bigr)-d_{\mathfrak p}(z,0),
\feqn
where the equalities are due to the isometry property of M\"{o}bius transformations.
\end{proof}
For an example of applications of Lemmas~\ref{bipo} and~\ref{zero} see Proposition~\ref{plln} below.
\section{Convergence to the boundary of the disk}
\label{converge}
In the this section we study the convergence of $U_n(z)$ to the disk boundary $\partial \cald,$ more precisely to the pair set $\cala=\{-\alpha,\alpha\}.$ Heuristically, $\lim_{n\to\infty}|\omega_n|=+\infty$ in probability (almost surely in a transient regime) for standard models of random walk on $\rr.$ Hence, in a suitable mode of convergence, $\cala$ is typically an attractor for $U_n$ (for instance, by Proposition~\ref{pbuz} above). The goal of this section is to characterize this convergence in terms of basic properties of the random walk $\omega_n.$
\par
First we impose the following basic assumption on the sequence $(x_n)_{n\in\nn}.$
\begin{assume}
\label{a1}
\item[(A1)] $(x_n)_{n\in\nn}$ is a stationary and ergodic sequence.
\item [(A2)] $P\bigl(x_1\in (-1,1)\bigr)=1.$
\item [(A3)] Let $\gamma_n=\log \bigl(\frac{1+x_n}{1-x_n}\bigr).$ Then $E(\gamma_1)$ is well defined, possibly infinite.
\end{assume}
We denote
\beqn
\label{eps}
\veps_x=\sign\bigl(E(\gamma_1)\bigr).
\feqn
It follows from a general result of \cite{keym} for a composition
of random M\"{o}bius transformations that if Assumption~\ref{a1} holds and $\veps_x\neq 0$ then, with probability one, for any $z\in\ol \cald\backslash\{-\alpha,\alpha\},$
\beqn
\label{thelimit}
\lim_{n\to\infty}U_n(z)=\veps_x \alpha.
\feqn
Formally, \eqref{thelimit} is proved in \cite{keym} only for i.\,i.\,d. sequences $x_n,$ but it is straightforward to see that their argument works verbatim under Assumption~\ref{a1}. The rate of escape to the boundary can be quantified using results from the previous section. Namely, by virtue of Lemma~\ref{bipo} and Proposition~\ref{pbuz} we have the following:
\begin{theorem}
\label{escape}
Let Assumption~\ref{a1} hold and assume in addition that $\veps_x\neq 0.$ Then w.\,p.\,1,
\beq
\lim_{n\to\infty} \frac{1}{n}B_{\veps_x\alpha}\bigl(U_n(z)\bigr)=-\bigl|E(\gamma_1)\bigr|
\quad \mbox{\rm and}\quad \lim_{n\to\infty} \frac{1}{n}d_{\mathfrak p}\bigl(0,U_n(z)\bigr)=\bigl|E(\gamma_1)\bigr|
\feq
for all $z\in\cald.$
\end{theorem}
Our next result gives the rate of convergence in terms of the Euclidean distance. For $\delta >0$ and $z\in\cc$ let $N_\delta(z)=\{w\in\cc:|z-w|<\delta\}.$ Denote $\ol \cald_\delta:=\ol \cald\backslash \bigl(N_\delta(\alpha) \cap N_\delta(-\alpha)\bigr).$ We have:
\begin{theorem}
\label{ucon}
Suppose that Assumption~\ref{a1} is satisfied. If $\veps_x\neq 0,$ then for any $\delta>0,$ the following holds w.\,p.\,1:
\beq
\lim_{n\to\infty} \frac{1}{n} \log \Bigl(\sup_{z\in \ol \cald_\delta} \bigl|U_n(z)-\veps_x\alpha\bigr|\Bigr)=-\bigl|E(\gamma_1)\bigr|.
\feq
\end{theorem}
\begin{remark}
\label{without}
Without the uniformity of the convergence, the result in Theorem~\ref{ucon} follows directly from \eqref{thelimit} and \eqref{curious}.We give below a short self-contained proof that doesn't assume \eqref{thelimit} a-priori and, more importantly, yields the uniformity. In addition, the calculation can be generalized readily to include other modes of convergence besides the almost sure one. See, for instance, Theorem~\ref{ucon1} and Proposition~\ref{converge} below.
\end{remark}
\begin{proof}[Proof of Theorem~\ref{ucon}]
For $\veps\in\{-1,1\}$ write
\beqn
\label{dist}
U_n(z)-\veps\alpha=\frac{z+y_n\alpha-\veps(\alpha+y_nz)}{1+y_n\bar \alpha z}=\frac{(z-\veps \alpha)(1-\veps y_n)}{1+y_n\bar \alpha z}.
\feqn
Birkhoff's ergodic theorem implies that
\beq
\lim_{n\to\infty} \omega_n=\veps_x\cdot \infty,\qquad \as
\feq
It follows then from \eqref{yg} that
\beqn
\label{ylim}
\lim_{n\to\infty} y_n=\veps_x, \qquad \mbox{\rm a.\,s.}
\feqn and hence by virtue of \eqref{dist}, the convergence in \eqref{thelimit} holds for all $z\in \ol \cald \backslash \{-\alpha,\alpha\}.$ Furthermore, \eqref{ylim}, \eqref{dist} and \eqref{yg} imply that with probability one, uniformly on $\cald_\delta$ for any $\delta>0,$
\beqn
\nonumber
\lim_{n\to\infty} \frac{1}{n}\log |U_n(z)-\veps_x\alpha|&=&\lim_{n\to\infty} \frac{1}{n}\log(1-\veps_xy_n)=-\lim_{n\to\infty} \frac{1}{n}\log(1+\veps_x\omega_n)
\\
&=&
\label{distr}
-\veps_x\cdot \Bigl(\lim_{n\to\infty} \frac{1}{n}\log \omega_n\Bigr) =-|E(\gamma_1)|,
\feqn
where in the last step we used \eqref{rrw} and the ergodic theorem.
\end{proof}
Recall $G_\alpha$ from \eqref{ga}. One can consider $G_\alpha$ as a topological group by equipping it with the topology of compact convergence on the complete metric space $(\cald,d_{\mathfrak p})$ (which is the same as the topology of compact convergence on $\cald$ equipped with the usual planar topology). A natural compactification of $G_\alpha$ is obtained by adding to it two elements:
\beqn
\label{compa}
g_{+\infty}:=q_{1,\alpha}=\frac{z+\alpha}{1+\ol \alpha z}=\alpha\qquad \mbox{\rm and}\qquad g_{-\infty}:=q_{-1,\alpha}=\frac{z-\alpha}{1-\ol \alpha z}=-\alpha.
\feqn
Let $\calk$ and $\ol \calk$ be, respectively, the sets of all compact (closed and bounded in the usual planar topology) subsets of $\cald$ and $\ol \cald,$ and
\beq
\label{ka}
\calk_\alpha:=\bigl\{E\in\ol \calk: \{-\alpha,\alpha\}\cap E=\emptyset\bigr\}
\feq
By definition, for any $x\in \rr$ and a sequence of reals $x_n,$ $n\in\nn,$
\beqn
\label{conv}
\lim_{n\to\infty} g_{x_n}=g_x \qquad \mbox{\rm iff}\qquad \forall E\in \calk,~\lim_{n\to\infty}\sup_{z\in E} |g_{x_n}(z)-g_x(z)|=0.
\feqn
It is easy to check (see, for instance, \eqref{dist} above) that
\beqn
\label{bineq}
\lim_{n\to\infty} g_{x_n}=g_x \qquad \mbox{\rm iff}\qquad \lim_{n\to\infty} x_n=x.
\feqn
Thus the topology of $G_\alpha$ is inherited from $\rr.$ To justify the definitions in \eqref{compa} observe that the equivalence in \eqref{bineq} still holds true for $x\in\{-\infty,+\infty\}$ if \eqref{conv}  is understood as the definition of the convergence in $\ol G_\alpha$ with $\calk$ replaced by $\calk_\alpha.$
\begin{corollary}
\label{train}
Under the conditions of Theorem~\ref{ucon}, $\lim_{n\to\infty} U_n=g_{+\infty \cdot \veps_x}$ in the compact convergence topology induced by $\ol\cald\backslash \{-\alpha,\alpha\}.$
\end{corollary}
Recall circles $H_z$ from \eqref{circle} and bipolar coordinate $\tau(z)$ from \eqref{bipolar}. In view of Lemmas~\ref{cinv} and~\ref{bipo}, assuming that $P(\gamma=0)\neq 1,$ we refer to the process $\bigl(U_n(z)\bigr)_{n\in\zz_+}$ restricted to $H_z$ as \emph{oscillating at infinity} if with probability one,
\beq
\limsup_{n\to\infty}\, \tau\bigl(U_n(z)\bigr)=-\liminf_{n\to\infty}\,\tau\bigl(U_n(z)\bigr) =+\infty.
\feq
We call the process restricted to $H_z$ \emph{transient to infinity} if with probability one,
\beq
\mbox{\rm either}\quad \lim_{n\to\infty} \tau\bigl(U_n(z)\bigr)=+\infty \quad \mbox{\rm or} \quad \lim_{n\to\infty} \tau\bigl(U_n(z)\bigr)=-\infty.
\feq
The following proposition asserts a standard random walk dichotomy for $U_n(z).$ Namely, it implies that for all $z\in\cald\backslash\{-\alpha,\alpha\},$ the process $U_n(z)$ is oscillating at at infinity if and only if $\omega_n$ on $\rr$ is oscillating, and is transient to infinity otherwise.
\begin{proposition}
\label{rec}
Let Assumption~\ref{a1} hold and assume in addition that $\veps_x=0.$ Then for any $z\in \cald\backslash\{-\alpha,\alpha\},$ the random walk $U_n(z)$ restricted to $H_z$ is oscillating at infinity.
\end{proposition}
\begin{proof}[Proof of Proposition~\ref{rec}]
If Assumption~\ref{a1} holds with $\veps_x=0,$ then the (generalized) random walk $\omega_n$ is recurrent \cite{ecycles}. Thus the proposition is a direct implication of Lemma~\ref{cinv} combined together with the result of Lemma~\ref{bipo}.
\end{proof}
We next consider the following general assumption:
\begin{assume}
\label{a3}
There exists a sequence of positive reals $(a_n)_{n\in\nn}$ and a random variable $W$ such that $\lim_{n\to\infty} a_n=+\infty$ and $\frac{\omega_n}{a_n}$ converges to $W$ in distribution as $n\to\infty.$
\end{assume}
In the above assumption we do not exclude the case that $W$ is degenerate. The most basic setup when Assumption~\ref{a3} holds true is described by the classical central limit theorem. If $(x_n)_{n\in\nn}$ is an i.\,i.\,d. sequence, $E(\gamma_1)=0$ and $E(\gamma_1^2)\in (0,\infty),$ then $W$ is a normal distribution and one can choose $a_n=\sqrt{n}.$ The CLT holds in fact in a functional form and even in that form can be extended to martingales, functionals of Markov chains, mixing and other ``weakly dependent in long range" sequences (see, for instance, \cite{clt4, mclt, clt5} and references therein). The theorem can also be extended to sequences with infinite variance in the domain of attraction of a stable law $W$ (see, for instance, \cite{borovs} and references therein).
\par
Let $F_W$ denote the distribution function of $W.$ It follows from \eqref{curious} that
\beqn
\label{tauclt}
\frac{1}{a_n}\log \Bigl|\frac{U_n(z)+\alpha}{U_n(z)-\alpha}\Bigr| \Longrightarrow W,
\feqn
where $\Rightarrow$ indicates the convergence in distribution. Our next next result is a refinement of \eqref{tauclt} that treats the numerator and denominator separately.
\begin{theorem}
\label{ucon1}
Let Assumption~\ref{a3} hold. Then for any $z\in\ol \cald\backslash\{-\alpha,\alpha\}$ and a real number $s>0$ we have
\beqn
\label{1}
\lim_{n\to\infty} P\Bigl(\frac{1}{a_n}\log |U_n(z)+\alpha|<-s\Bigr)=F_W(-s)
\feqn
and
\beqn
\label{3}
\lim_{n\to\infty} P\Bigl(\frac{1}{a_n}\log |U_n(z)-\alpha|<-s\Bigr)=1-F_W(s).
\feqn
\end{theorem}
The theorem is a counterpart of Theorem~\ref{ucon} under Assumption~\ref{a3}. We omit a formal proof of the theorem because it can be readily obtained through a suitable modification of the almost sure computation in \eqref{distr}. Indeed, \eqref{distr} is in essence an observation that by virtue of \eqref{dist} and \eqref{ylim}, $\bigl|U_n(z)+\veps\alpha\bigr|$ can only be small if $1-\veps y_n$ is small, in which case $\bigl|U_n(z)+\veps\alpha\bigr|\sim c_z(1-\veps y_n)$ as $n\to\infty,$ where $c_z>0$ is a constant that depends on $z.$ Here and henceforth the notation $u_n\sim v_n$ for two sequence of real numbers $(u_n)_{n\in\nn}$ and $(v_n)_{n\in\nn}$ stands for $\lim_{n\to\infty} \frac{u_n}{v_n}=1.$
\par
We note that in the generic situation when Assumption~\ref{a3} holds true and $W$ is distributed according to a stable law, the following limit exists:
\beq
\chi:=\lim_{s\to\infty}\frac{1-F_w(s)}{1-F_W(s)+F_W(-s)}\in [0,1].
\feq
In view of \eqref{1} and \eqref{3} this implies:
\beq
\lim_{s\to\infty} \lim_{n\to\infty} \frac{P\bigl(\log |U_n(z)+\alpha|<-sa_n\bigr)}{P\bigl(\log |U_n(z)-\alpha|<-sa_n\bigr)}=\frac{1-\chi}{\chi}.
\feq
An important aspect of oscillations of one-dimensional random walks is captured by laws of iterated logarithm. For instance, we have
\begin{proposition}
\label{plln}
Suppose that $x_n,$ $n\in\nn,$ form an i.\,i.\,d. sequence, $E(\gamma_1)=0,$ and $\sigma:=\bigl(E(\gamma_1^2)\bigr)^{1/2}\in(0,\infty).$ Then, for any $z\in \cald\backslash\{-\alpha,\alpha\},$ with probability one,
\beq
\limsup_{n\to\infty}\frac{B_\alpha\bigl(U_n(z)\bigr)}{\phi(n)}=
\limsup_{n\to\infty}\frac{B_{-\alpha}\bigl(U_n(z)\bigr)}{\phi(n)}=\limsup_{n\to\infty}\frac{d_{\mathfrak p}\bigl(0,U_n(z)\bigr)}{\phi(n)}=\sigma,
\feq
where $\phi(n):=\sqrt{2\pi n\log \log n}.$
\end{proposition}
We remark that the usual law iterated logarithm can be extended to sums of weakly dependent random variables (see, for instance, \cite{lln} and references therein) and (in a different form) to sequences with infinite variance \cite{borovs}. Many interesting almost sure results about simple random walk can be found in \cite{budapest}.
\section{A random walk on orthogonal pencils of circles}
\label{apencils}
In this section we consider a random motion on $\ol \cald$ which is intimately related to the evolution of the sequence $U_n(z)$ that we studied so far. The underlying idea is to adopt a pattern which is regular for the transition mechanism of random walks on $\zz^2$ and consider a discrete-time motion where two orthogonal coordinates changes in an alternating fashion, one at a time. The natural choice of the underlying coordinates within the context is the pair of bipolar coordinates $(\sigma,\tau)$ introduced in \eqref{bipolar}. We discuss here one, arguably the most simply related to $\bigl(U_n(z)\bigr)_{n\in\zz_+}$ example of such a motion. Several other examples will be discussed by the authors elsewhere \cite{ranew}.
\par
Before we proceed with the specific example, we recall that isoparametric curves in bipolar coordinates (i.\,e., curves along which only one coordinate in the pair $(\sigma,\tau)$ changes while another remains constant) constitute circles. More specifically, $\sigma$ remains constant on circles $H_z$ introduced in Lemma~\ref{cinv} (dashed circles in Fig.~\ref{fig2} below) whereas $\tau$ is constant on circles with centers lying on the line $l_\alpha,$ each one being orthogonal to every circle in the collection $(H_z)_{z\in\ol \cald}$ (dotted circles in Fig.~\ref{fig2}). In particular, circular arcs corresponding to constant $\tau$ are geodesic lines in the Poincar\'{e} disk model. The family of circles include two degenerate circles, namely the line $l_\alpha$ corresponding to $\sigma=\pi$ and an orthogonal to it which corresponds to $\tau=0.$ The family of coaxial circles corresponding to a constant $\sigma$ are sometimes referred as to an elliptic pencil of (Apollonian \cite{apollo}) circles. Similarly, the circles corresponding to a constant value of $\tau$ form a hyperbolic pencil \cite{amazing,localv,pencils}. The terminology is related to a representation of planar circles as projective lines in dimension 3.
\par
Recall the bipolar coordinates $(\sigma,\tau)$ from \eqref{bipolar}. Rather than working directly with $\sigma$ we will consider
\beq
\varsigma=\sign(\sigma)\cdot \pi-\sigma, \qquad \varsigma \in \Bigl[-\frac{\pi}{2},\frac{\pi}{2}\Bigr].
\feq
This version of the angular component is a continuous function of the Cartesian coordinates. Furthermore,
\beq
\frac{z+\alpha}{\alpha-z}=e^{\tau(z)+i\varsigma(z)},\qquad \forall~z\in\cald.
\feq
We remark that similarly customized bipolar coordinates are used, for instance, in \cite{bp1}.
\par
Fix any $p\in(0,1)$ and let $q:=1-p.$ Let $(u_n)_{n\in\nn}$ be a sequence of i.\,i.\,d. random variables, each one distributed uniformly over the interval $\bigl[-\frac{\pi}{2},\frac{\pi}{2}\bigr].$ Let $(c_n)_{n\in\zz}$ be a sequence of Bernoulli random variables (``coins") such that $P(c_n=1)=p$ and $P(c_n=0)=q.$ Assume that the sequences $(x_n)_{n\in\nn},$ $(u_n)_{n\in\nn},$ and $(c_n)_{n\in\nn}$ are independent each of other. Further, introduce the following  auxiliary functionals of the sequence  $c_n:$ $S_n=\sum_{k=1}^n c_k$ and $k_n=\max\{t\leq n:c_t=0\}.$ We will consider the following random walk $(Z_n)_{n\in\zz_+}$ on the state space $\ol \cald.$ Let $Z_n=(\varsigma_n,\tau_n)_{bipolar}$ denote the location of the random walk at time $n\in\zz_+$ in the bipolar coordinates $(\varsigma,\tau).$ At each instant of time $n\in\nn,$ the random walk changes its location from $Z_{n-1}$ to $Z_n$ according to the following rule and independently of past events:
\begin{itemize}
\item[-] If $c_n=1,$ then $Z_n=g_{\gamma_n}(Z_{n-1}),$ thus $\varsigma_n=\varsigma_{n-1}$ and $\tau_n=\tau_{n-1}+\gamma_{S_n}.$
\item[-] If $c_n=0,$ then $\tau_n=\tau_{n-1}$ and $\varsigma_n=u_n.$
\end{itemize}
Thus $Z_n$ is a Markov chain with transition kernel determined by
\beq
P\bigl(\varsigma_{n+1}\in A,\tau_{n+1}\in B\,\bigl|\,\varsigma_n,\tau_n\bigr)=
p\cdot {\bf 1}_{\{\varsigma_n\in A\}}\cdot P\bigl(g_{\gamma_1}(Z_{n-1})\in B\bigr)+q\cdot {\bf 1}_{\{\tau_n\in B\}}\cdot \frac{|A|}{\pi},
\feq
where $A\subset \bigl[-\frac{\pi}{2},\frac{\pi}{2}\bigr]$ and $B\subset \rr$ are Borel sets, $\odin_E$ denotes the indicator function of a set $E,$ and $|B|$ is the Lebesgue measure of the set $B.$  Furthermore, for $n\in\nn,$
\beq
(\varsigma_n,\tau_n)=c_n\cdot (\varsigma_{n-1},\tau_{n-1}+\gamma_{S_n})+(1-c_n)\cdot (u_n,\tau_n)
\feq
and hence
\beq
(\varsigma_{n},\tau_n)=(u_{k_n},\tau_0+\omega_{S_n}),\qquad n\in\nn.
\feq
In particular,
\beqn
\label{rw1}
Z_n=U_{S_n}({\mathfrak z}_n), \qquad \mbox{\rm where} \qquad {\mathfrak z}_n:=(\varsigma_n,\tau_n)_{bipolar}=(u_{k_n},\tau_0+\omega_{S_n})_{bipolar}.
\feqn
The last formula allows to apply the results of the previous sections to the random walk $Z_n.$ It turns out that even though, in contrast to $U_n(z),$ the distance $d_{\mathfrak p}(Z_n,0)$ between $Z_n$ and the origin can be arbitrarily large with a positive probability for any $n\in\nn$ independently of the value of $\omega_n,$  the asymptotic behavior of the sequence $Z_n$ is quite similar to that of $U_n$ (see Theorems~\ref{thm37} and~\ref{ucon3} and also Remark~\ref{rem35} below). In this sense, the former model of a random motion on $\ol \cald$ seems to offer an adequate generalization of the latter.
\par
\begin{figure}[!ht]
\begin{center}
\begin{tikzpicture}[scale =0.8]
\clip(-7,-4) rectangle (4,5);
\draw [line width=1pt] (-1.12,0.89) circle (4.013277961965755cm);
\draw[dashed] [line width=1pt] (-4.5432974634567795,-1.2046203657680743)-- (2.3075343970440345,2.9776800418359097);
\draw[dashed] [shift={(2.888431895775272,-5.676025934506512)},line width=1pt] plot[domain=1.637822793434394:2.59995036839282,variable=\t]({1*8.673181021344863*cos(\t r)+0*8.673181021344863*sin(\t r)},{0*8.673181021344863*cos(\t r)+1*8.673181021344863*sin(\t r)});
\draw[dashed] [shift={(0.3309117026156172,-1.5167803593575142)},line width=1pt]  plot[domain=1.166562704597435:3.077636788742166,variable=\t]({1*4.884194801215645*cos(\t r)+0*4.884194801215645*sin(\t r)},{0*4.884194801215645*cos(\t r)+1*4.884194801215645*sin(\t r)});
\draw[dashed] [shift={(-0.5558387159915673,-0.07467532154864191)},line width=1pt]  plot[domain=0.8264715007811323:3.417727992558468,variable=\t]({1*4.144466559847346*cos(\t r)+0*4.144466559847346*sin(\t r)},{0*4.144466559847346*cos(\t r)+1*4.144466559847346*sin(\t r)});
\draw[dashed] [shift={(-2.9084745154541007,3.751370730804398)},line width=1pt]  plot[domain=4.393760447518583:6.133624353000603,variable=\t]({1*5.218667820490378*cos(\t r)+0*5.218667820490378*sin(\t r)},{0*5.218667820490378*cos(\t r)+1*5.218667820490378*sin(\t r)});
\draw[dashed] [shift={(-1.7541942777881347,1.8741871277835132)},line width=1pt]  plot[domain=-2.3068635982392256:0.26787778439923976,variable=\t]({1*4.154293220591969*cos(\t r)+0*4.154293220591969*sin(\t r)},{0*4.154293220591969*cos(\t r)+1*4.154293220591969*sin(\t r)});
\draw[dotted] [shift={(-6.865410056788786,-2.246793069183498)},line width=1pt]  plot[domain=-0.16020842873872088:1.1596885902846057,variable=\t]({1*5.13990934404543*cos(\t r)+0*5.13990934404543*sin(\t r)},{0*5.13990934404543*cos(\t r)+1*5.13990934404543*sin(\t r)});
\draw[dotted] [shift={(144.40497742708402,78.57719125607755)},line width=1pt]  plot[domain=3.6076161548313506:3.6562740347475775,variable=\t]({1*164.96308296225456*cos(\t r)+0*164.96308296225456*sin(\t r)},{0*164.96308296225456*cos(\t r)+1*164.96308296225456*sin(\t r)});
\draw[dotted] [shift={(-5.835374337263382,-1.6693357214578461)},line width=1pt]  plot[domain=-0.3413313931989901:1.3358730295534569,variable=\t]({1*3.1561208595190156*cos(\t r)+0*3.1561208595190156*sin(\t r)},{0*3.1561208595190156*cos(\t r)+1*3.1561208595190156*sin(\t r)});
\draw[dotted] [shift={(6.276037434770659,5.333100386462109)},line width=1pt]  plot[domain=3.200518505652849:4.164594647381183,variable=\t]({1*7.9732138072121685*cos(\t r)+0*7.9732138072121685*sin(\t r)},{0*7.9732138072121685*cos(\t r)+1*7.9732138072121685*sin(\t r)});
\draw[dotted] [shift={(-11.013438819784929,-4.354498954772841)},line width=1pt]  plot[domain=0.12209367085772392:0.8527775761337828,variable=\t]({1*10.778138359322298*cos(\t r)+0*10.778138359322298*sin(\t r)},{0*10.778138359322298*cos(\t r)+1*10.778138359322298*sin(\t r)});
\draw[dotted] [shift={(-7.759737328241321,-2.6569697666132384)},line width=1pt]  plot[domain=-0.07031859914041672:1.0515874226105562,variable=\t]({1*6.6361213011589575*cos(\t r)+0*6.6361213011589575*sin(\t r)},{0*6.6361213011589575*cos(\t r)+1*6.6361213011589575*sin(\t r)});
\begin{scriptsize}
\draw [fill=black] (-1.12,0.89) circle (2.5pt);
\draw[color=black] (-0.72,0.86) node {$O$};
\draw [fill=black] (-4.5432974634567795,-1.2046203657680743) circle (0.5pt);
\draw[color=black] (-4.5432974634567795,-1.2046203657680743) node[below left] {$-\alpha$};
\draw [fill=black] (2.3075343970440345,2.9776800418359097) circle (0.5pt);
\draw[color=black] (2.3075343970440345,2.9776800418359097) node[above right] {$\alpha$};
\draw [fill=black] (-2.9434145567957066,0.743758795735773) circle (2.5pt);
\draw[color=black] (-3.28,1.08) node {$z$};
\draw [->,line width=.5pt] (-3.03,0.67) -- (-2.8887459673380973,0.8417715083982361);
\draw [->,line width=0.5pt] (-2.8887459673380973,0.8417715083982361) -- (-3.122233165183343,1.1316907591040475);
\draw [->,line width=0.5pt] (-3.122233165183343,1.1316907591040475) -- (-3.4504620698937623,1.4766489564854823);
\draw [->,line width=0.5pt] (-3.4504620698937623,1.4766489564854823) -- (-3.549389111261042,1.56894687013822);
\draw [->,line width=0.5pt] (-3.549389111261042,1.56894687013822) -- (-3.4382205085233055,1.6974717188426558);
\draw [->,line width=0.5pt] (-3.4382205085233055,1.6974717188426558) -- (-3.1524975186088247,1.9911610427359567);
\draw [->,line width=0.5pt] (-3.1524975186088247,1.9911610427359567) -- (-2.854027261095325,2.2517562540889147);
\draw [->,line width=0.5pt] (-2.854027261095325,2.2517562540889147) -- (-2.713934169893444,2.3604471548657555);
\draw [->,line width=0.5pt] (-2.713934169893444,2.3604471548657555) -- (-3.0051246627998545,2.711566291021777);
\draw [->,line width=0.5pt] (-3.0051246627998545,2.711566291021777) -- (-3.331759608783959,3.0721058659235028);
\draw [->,line width=0.5pt] (-3.331759608783959,3.0721058659235028) -- (-2.9588485591786635,3.3500864055136415);
\draw [->,line width=0.5pt] (-2.9588485591786635,3.3500864055136415) -- (-2.6393607125940366,3.541565591906161);
\draw [->,line width=0.5pt] (-2.6393607125940366,3.541565591906161) -- (-2.3261189338452084,2.9315337628319043);
\draw [->,line width=0.5pt] (-2.3261189338452084,2.9315337628319043) -- (-2.2037590788289316,2.69492342106163);
\draw [->,line width=0.5pt] (-2.2037590788289316,2.69492342106163) -- (-2.000435412575939,2.3038082534179978);
\draw [->,line width=0.5pt] (-2.000435412575939,2.3038082534179978) -- (-1.685929288241317,1.7037995200977665);
\draw [->,line width=0.5pt] (-1.685929288241317,1.7037995200977665) -- (-2.070643091610673,1.4589940435591726);
\draw [->,line width=0.5pt] (-2.070643091610673,1.4589940435591726) -- (-1.929289493859791,1.2315063895244087);
\draw [->,line width=0.5pt] (-1.929289493859791,1.2315063895244087) -- (-1.7234317899384504,0.8756807940541318);
\draw [->,line width=0.5pt] (-1.7234317899384504,0.8756807940541318) -- (-1.6173301050101725,0.679159444261243);
\draw [->,line width=0.5pt] (-1.7234317899384504,0.8756807940541318) -- (-1.2307130435191063,-0.13608053796333053);
\draw [->,line width=0.5pt] (-1.2307130435191063,-0.13608053796333053) -- (-0.8735086228963354,-1.1055837551966954);
\draw [->,line width=0.5pt] (-0.8735086228963354,-1.1055837551966954) -- (-0.48495468740027636,-0.9154441428170443);
\draw [->,line width=0.5pt] (-0.48495468740027636,-0.9154441428170443) -- (-0.04411133087790642,-0.6453761700495795);

\end{scriptsize}
\end{tikzpicture}
\caption{
Random walk on orthogonal pencils of circles. Every dotted circle intersects every dashed circle at a right angle. Every dotted circle corresponds to some $\tau \in \rr$ and every dashed circle corresponds to some $\varsigma \in [-\frac{\pi}{2},\frac{\pi}{2}]$. The arrows indicate a random trajectory $\{U_n(z):n=0,\ldots,24\}$ obtained through numerical simulation.  }
\label{fig2}
\end{center}
\end{figure}
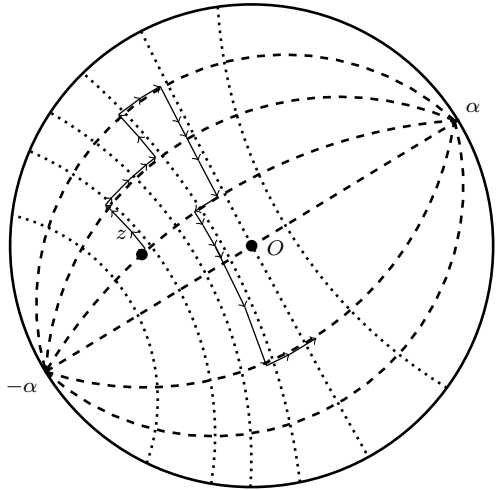
\par
Our main results in this section (Theorem~\ref{thm37} and Theorem~\ref{ucon3} below) are not affected by the assumption on the position or distribution of the initial value $\sigma_0$ (basically, because $\varsigma_n=\varsigma_{k_n}$ would be anyway distributed uniformly over $\bigr[-\frac{\pi}{2},\frac{\pi}{2}\bigr]$ for all $n$ large enough, so that $k_n\geq 1$). For simplicity, we will assume throughout the rest of the paper that $\varsigma_n$ are distributed uniformly over $\bigr[-\frac{\pi}{2},\frac{\pi}{2}\bigr]$ for all $n\in\zz_+,$ namely:
\begin{assume}
\label{33}
$\varsigma_0$ is distributed uniformly over $\bigr[-\frac{\pi}{2},\frac{\pi}{2}\bigr].$
\end{assume}
We note that the randomness of $S_n$ in \eqref{r1} is a simple ``removable" technicality since
\beqn
\label{r}
Z_n=\frac{{\mathfrak z}_n+\alpha y_{S_n}}{1+
\ol \alpha y_{S_n}{\mathfrak z}_n}=\frac{{\mathfrak z}_n+\alpha \witi y_n}{1+
\ol \alpha \witi y_n{\mathfrak z}_n},
\feqn
where, similarly to \eqref{yg},
\beqn
\label{r3}
\witi y_n:=y_{S_n}=\frac{e^{\witi \omega_n}-1}{e^{\witi \omega_n}+1} \qquad \mbox{\rm with}
\qquad \witi \omega_n:=\sum_{k=1}^n \gamma_k\cdot {\bf 1}_{\{S_k=S_{k-1}+1\}} ,
\feqn
and $\witi \gamma_k:=\gamma_k\cdot {\bf 1}_{\{S_k=S_{k-1}+1\}}$ are i.\,i.\,d. random variables. To illustrate the technicality arising from the randomness of ${\mathfrak z}_n$ in \eqref{rw1} and \eqref{r} we will prove the following almost sure result, an analogue of Theorem~\ref{escape} and Proposition~\ref{plln} for the sequence $Z_n.$
\begin{theorem}
\label{thm37}
Let $(Z_n)_{n\in\zz_+}$ be a random walk defined in \eqref{rw1}. Then for any $\tau_0\in \rr,$
\label{rwt}
\begin{enumerate}[(i)]
\item Under Assumption~\ref{a1},
with probability one,
\beq
\lim_{n\to\infty} \frac{1}{n}B_{\veps_x\alpha}\bigl(Z_n\bigr)=-p\cdot \bigl|E(\gamma_1)\bigr|
\quad \mbox{\rm and}\quad \lim_{n\to\infty} \frac{1}{n}d_{\mathfrak p}\bigl(0,Z_n\bigr)=p\cdot \bigl|E(\gamma_1)\bigr|.
\feq
\item Under the conditions of Proposition~\ref{plln}, with probability one,
\beq
\limsup_{n\to\infty}\frac{B_\alpha\bigl(U_n(z)\bigr)}{\phi(n)}=
\limsup_{n\to\infty}\frac{B_{-\alpha}\bigl(U_n(z)\bigr)}{\phi(n)}=\limsup_{n\to\infty}\frac{d_{\mathfrak p}\bigl(0,U_n(z)\bigr)}{\phi(n)}=\sigma\sqrt{p},
\feq
where $\sigma=\bigl(E(\gamma_1^2)\bigr)^{1/2}\in(0,\infty)$ and $\phi(n):=\sqrt{2\pi n\log \log n}.$
\end{enumerate}
\end{theorem}
\begin{proof}[Proof of Theorem~\ref{rwt}]
$\mbox{}$
\item[\emph{(i)}] The result will follow from \eqref{r} and \eqref{r3} along with the results in Proposition~\ref{pbuz} and Lemma~\ref{zero} provided that we are able to show that the random initial condition doesn't interfere and destroy the almost sure convergence, namely that w.\,p.\,1, \beqn
\label{show}
\lim_{n\to\infty} \frac{1}{n}B_{\veps_x\alpha}\bigl({\mathfrak z}_n\bigr)=\lim_{n\to\infty} \frac{1}{n}d_{\mathfrak p}\bigl(0,{\mathfrak z}_n\bigr)=0.
\feqn
Using the explicit formulas for the Poincar\'{e} distance and Buzemann functions, it is easy to check that \eqref{show} amounts to
\beq
\lim_{n\to\infty} \frac{1}{n}\log (1-|{\mathfrak z}_n|)=0\qquad \mbox{\rm a.\,s.}
\feq
To facilitate the proof of the second part of the theorem we will show a more general
\beqn
\label{lim7}
\lim_{n\to\infty} \frac{1}{n^\beta}\log (1-|{\mathfrak z}_n|)=0\qquad \mbox{\rm a.\,s.}
\feqn
for any $\beta>0.$ Toward this end observe that by virtue of \eqref{bpz} and \eqref{bpxy}, in the Cartesian coordinates we have
\beq
{\mathfrak z}_n\ol \alpha=i\cot \Bigl(\frac{\varsigma_n}{2}+i\frac{\tau_0}{2}\Bigr)=\frac{\sinh \tau_0}{\cosh \tau_0+\cos \varsigma_n}\pm i\frac{\sin \varsigma_n}{\cosh \tau_0+\cos \varsigma_n},
\feq
and hence
\beq
|{\mathfrak z}_n|^2&=&
\frac{\sin^2\varsigma_n+\sinh^2 \tau_0}{(\cosh \tau_0+\cos\varsigma_n)^2} .
\feq
Therefore,
\beq
1-|{\mathfrak z}_n|^2
=\frac{2\cos \varsigma_n}{\cosh \tau_0+\cos\varsigma_n}\geq \frac{2\cos \varsigma_n}{\cosh \tau_0+1} .
\feq
Thus, in order to prove \eqref{lim7} it suffices to show that, w.\,p.1
\beq
\lim_{n\to\infty} \frac{1}{n^\beta} \log \cos \varsigma_n=0,
\feq
or, equivalently,
\beq
P\Bigl(\frac{1}{n^\beta} \log \cos \varsigma_n<-\veps~\mbox{\rm i.\,o.}\Bigr)=0,\qquad \forall~\veps>0.
\feq
By the Borel-Cantelli lemma, the last assertion is implied by the following one
\beq
\sum_{n=1}^\infty P\bigl((-\log \cos \varsigma_n)^{1/\beta} >\delta n\bigr),\qquad \forall~\delta>0,
\feq
which in turn is equivalent to
\beqn
\label{esigma-11}
E\bigl((-\log \cos \varsigma_n)^{1/\beta}\bigr)<\infty.
\feqn
Let $s_n=\pi/2-|\varsigma_n|$ and let $\ol s\in (0,\pi/2)$ be such that $\sin \ol s=\frac{\ol s}{2}.$ Then  $\sin s_n>\frac{\ol s}{2}$ for $s_n\in (0,\ol s),$ and hence
\beq
E\bigl((-\log \cos \varsigma_n)^{1/\beta}\bigr)
=E\bigl((-\log \sin s_n)^{1/\beta}\bigr)<\infty.
\feq
if $\int_0^{\ol s} (-\log s)^{1/\beta}ds <\infty.$ Since the latter inequality holds true, so is \eqref{lim7}. The proof of part (i) of the theorem is complete.
\item[\emph{(ii)}]
The proof of part (ii) is similar. Indeed, in view of \eqref{r} and \eqref{r3} along with the results in Proposition~\ref{pbuz} and Lemma~\ref{zero} we only need to show that w.\,p.\,1,
\beq
\lim_{n\to\infty} \frac{1}{\phi(n)}B_{-\veps_x\alpha}\bigl({\mathfrak z}_n\bigr)=\lim_{n\to\infty} \frac{1}{\phi(n)}d_{\mathfrak p}\bigl(0,{\mathfrak z}_n\bigr)=0.
\feq
The latter identities follow from \eqref{lim7} with any $\beta<1/2.$
\end{proof}
\begin{remark}
\label{rem35}
We remark that the conclusions of Theorem~\ref{thm37} remains true for any distribution of i.\,i.\,d. random variables $\varsigma_n$ as long as the condition \eqref{esigma-11} is satisfied.
\end{remark}
Recall $\veps_x$ from \eqref{eps}. We conclude with a result showing that the assertions of Theorems~\ref{ucon} and~\ref{ucon1} remain true for the sequence $Z_n$ in place of $U_n(z).$ More specifically, we have the following:
\begin{theorem}
\label{ucon3}
$\mbox{}$
\begin{enumerate}[(i)]
\item Suppose that Assumption~\ref{a1} is satisfied. If $\veps_x\neq 0,$ then
\beq
\lim_{n\to\infty} \frac{1}{n} \log \bigl|Z_n-\veps_x\alpha\bigr|=-p\cdot \bigl|E(\gamma_1)\bigr|,\qquad \as
\feq
Moreover, the convergence is uniform on compact intervals in $\tau_0.$
\item Let Assumption~\ref{a3} hold. Then for all $\tau_0\in\rr$ and $s>0$ we have
\beq
\lim_{n\to\infty} P\Bigl(\frac{1}{a_n}\log |Z_n+\alpha|<-s\Bigr)=F_W(-s)
\feq
and
\beq
\lim_{n\to\infty} P\Bigl(\frac{1}{a_n}\log |Z_n-\alpha|<-s\Bigr)=1-F_W(s).
\feq
\end{enumerate}
\end{theorem}
\begin{proof}[Proof of Theorem~\ref{ucon3}]
Observe that for all $z\in\cald$ and $\veps\in\{-1,1\},$ we have (see Fig.~\ref{fig2}):
\beqn
\label{31}
|U_n(z_-)-\veps\alpha|\leq |U_n(z)-\veps\alpha|\leq |U_n(z_+)-\veps\alpha|,
\feqn
where given $z=(\varsigma_0,\tau_0)_{bipolar},$ we define $z_-=(0,\tau_0)_{bipolar}$ and $z_+=(\pi/2,\tau_0)_{bipolar}.$
To verify claims (i) and (i) in the theorem, apply, respectively, Theorem~\ref{ucon} and Theorem~\ref{ucon3}, taking in account \eqref{rw1} and \eqref{r3} along with \eqref{31}.
\end{proof}
\section{Numerical simulations of $Z_n$}
\label{simul}
This section contains results of the numerical simulation of the random walk $Z_n$ introduced in Section~\ref{apencils}. In all the scenarios reported below $\alpha$ is taken to be equal to 1, the number of steps is $n=300,000$ and $\varsigma_0$ is distributed uniformly over $\bigl[-\frac{\pi}{2},\frac{\pi}{2}\bigr]$  with parameter $p$ varying from case to case. Each figure below represents one random realization of the random walk with parameters as indicated above and further specified in the caption.
\par
In Fig.~\ref{fig3}, $x_n$ is distributed uniformly over the interval $(-1,1),$ and hence the underlying random motion is recurrent. When $p=0.5$ most of the observations are located near the poles $\pm \alpha.$ The more $p$ deviates from $0.5,$ the more the cluster of observations appears to have a structure resembling a disjoint union of circles from one dominant pencil (dashed lines in Fig.~\ref{fig0} if $p$ is close to one and dotted lines if $p$ is close to zero).
\\
In Fig.~\ref{fig4}, $x_n$ has a triangular density with mode $0.1$ supported on $(-1,1),$ namely in these three cases we have
\beq
P(x_n\in dx)=
\left\{
\begin{array}{lcc}
\frac{100}{121}(x+1)dx&\mbox{\rm if}&-1<x<\frac{21}{100},\\
$\mbox{}$&&
\\
\frac{100}{79}(1-x)dx&\mbox{\rm if}&\frac{21}{100}<x<1.
\end{array}
\right.
\feq
For this distribution, $\veps_x=E(\gamma_1) \approx 0.0781,$ and therefore the random walk is with probability one attracted to $\alpha$ by virtue of part (i) of Theorem~\ref{thm37} (since $\lim_{n\to\infty} B_\alpha (Z_n)=\infty$). It appears (see Fig.~\ref{fig4}(c) below) that if $p$ is large, then for an ``intermediately large" number of steps $n$ there is a ``thick" cluster present around the repelling fixed point $\alpha=-1.$ The explanation might be provided by \eqref{bipolar} and Lemma~\ref{bipo}, the fluctuations of $U_n$ measured in the usual Euclidean distance are much large in magnitude near the poles $\pm \alpha$ than they are at points of the trajectory located far away from the boundary.\\
\section*{Acknowledgment} 
We would like to thank the anonymous referee for their very helpful comments and suggestions.
\begin{figure}
\centering
\begin{tabular}{ccc}
\includegraphics[scale = 0.6]{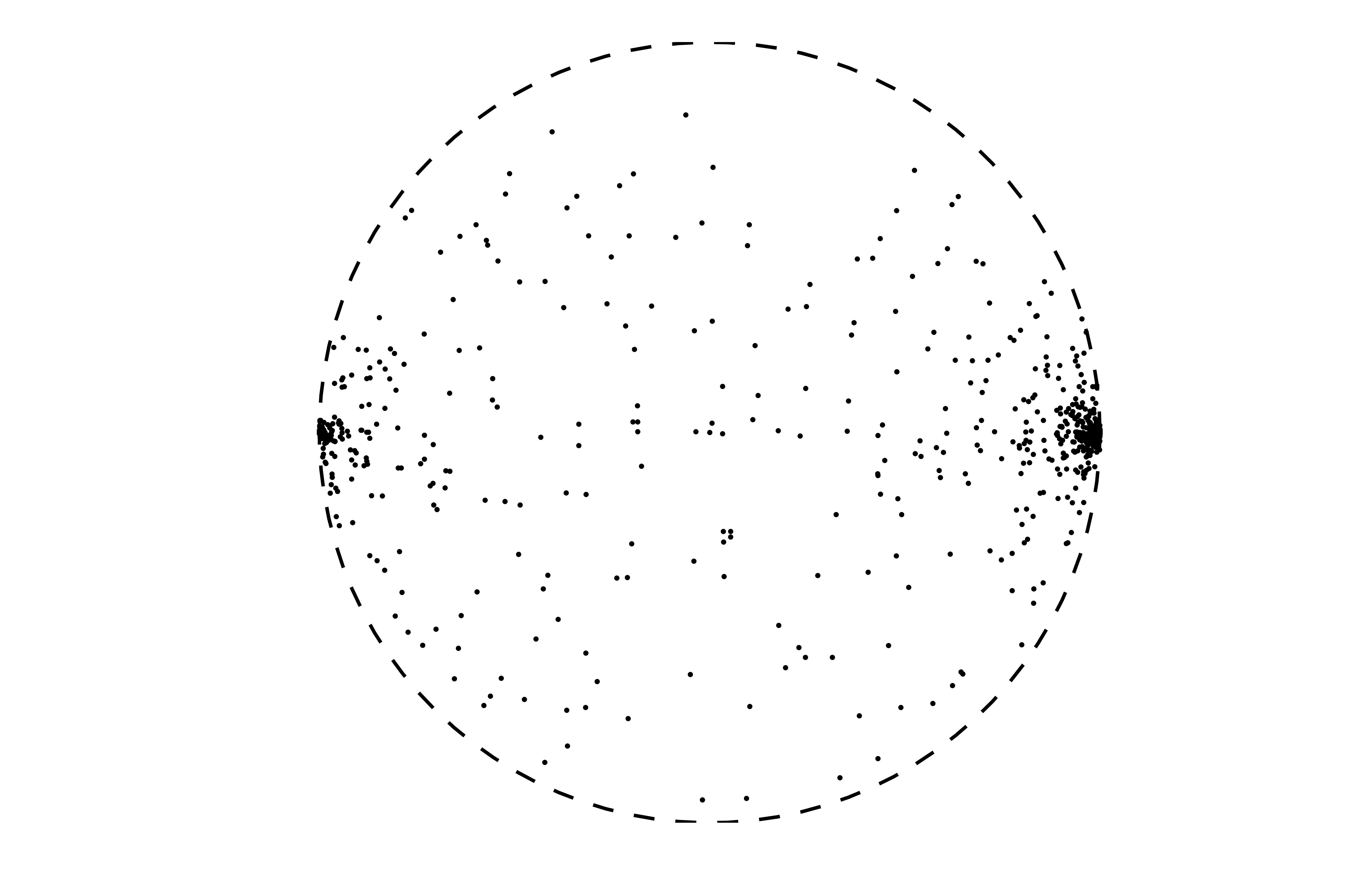} &
\includegraphics[scale = 0.6]{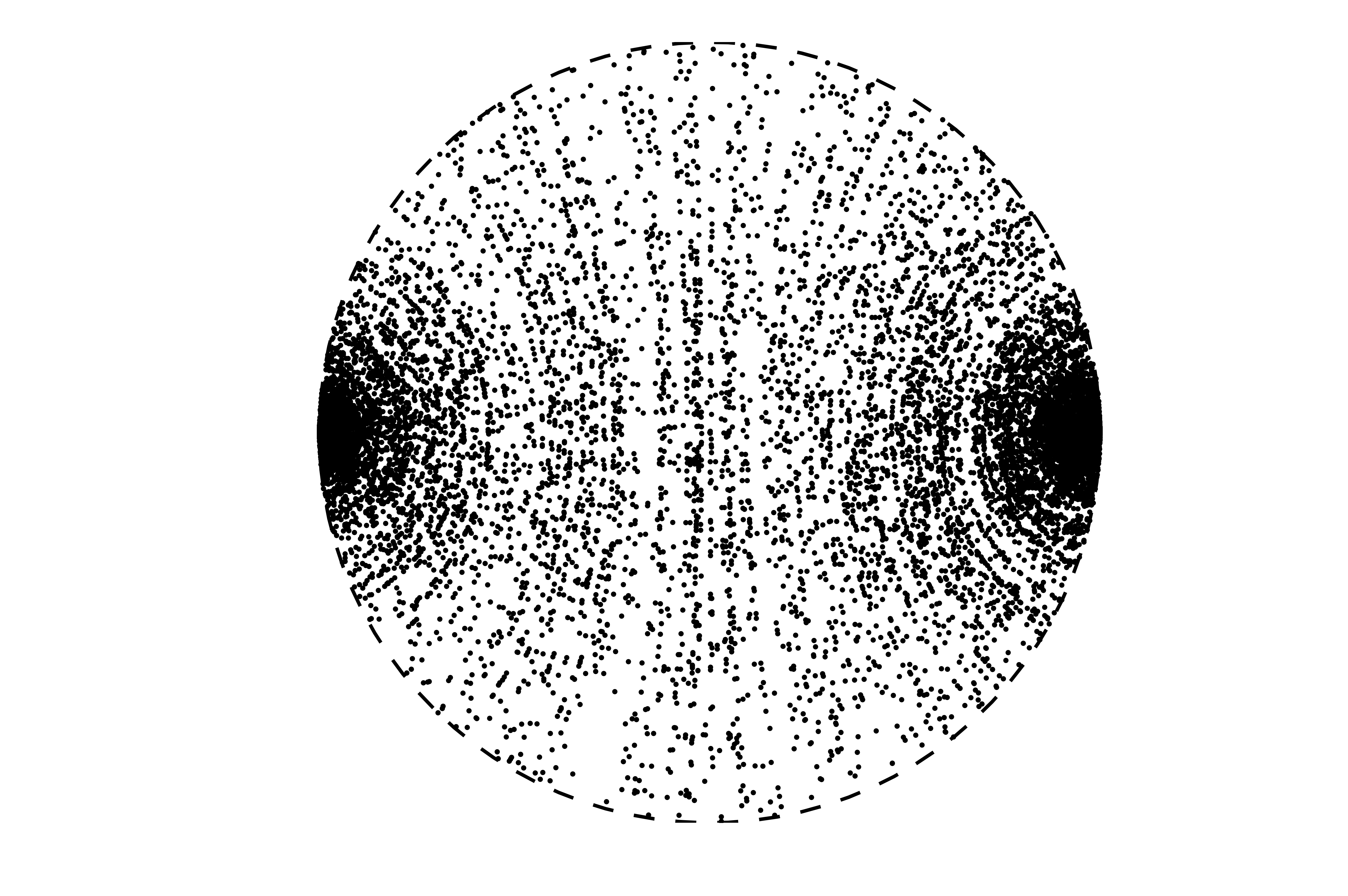} \\
(a) $p=0.5$ & (b) $p=0.1$
\end{tabular}
\begin{tabular}{c}
\includegraphics[scale = 0.6]{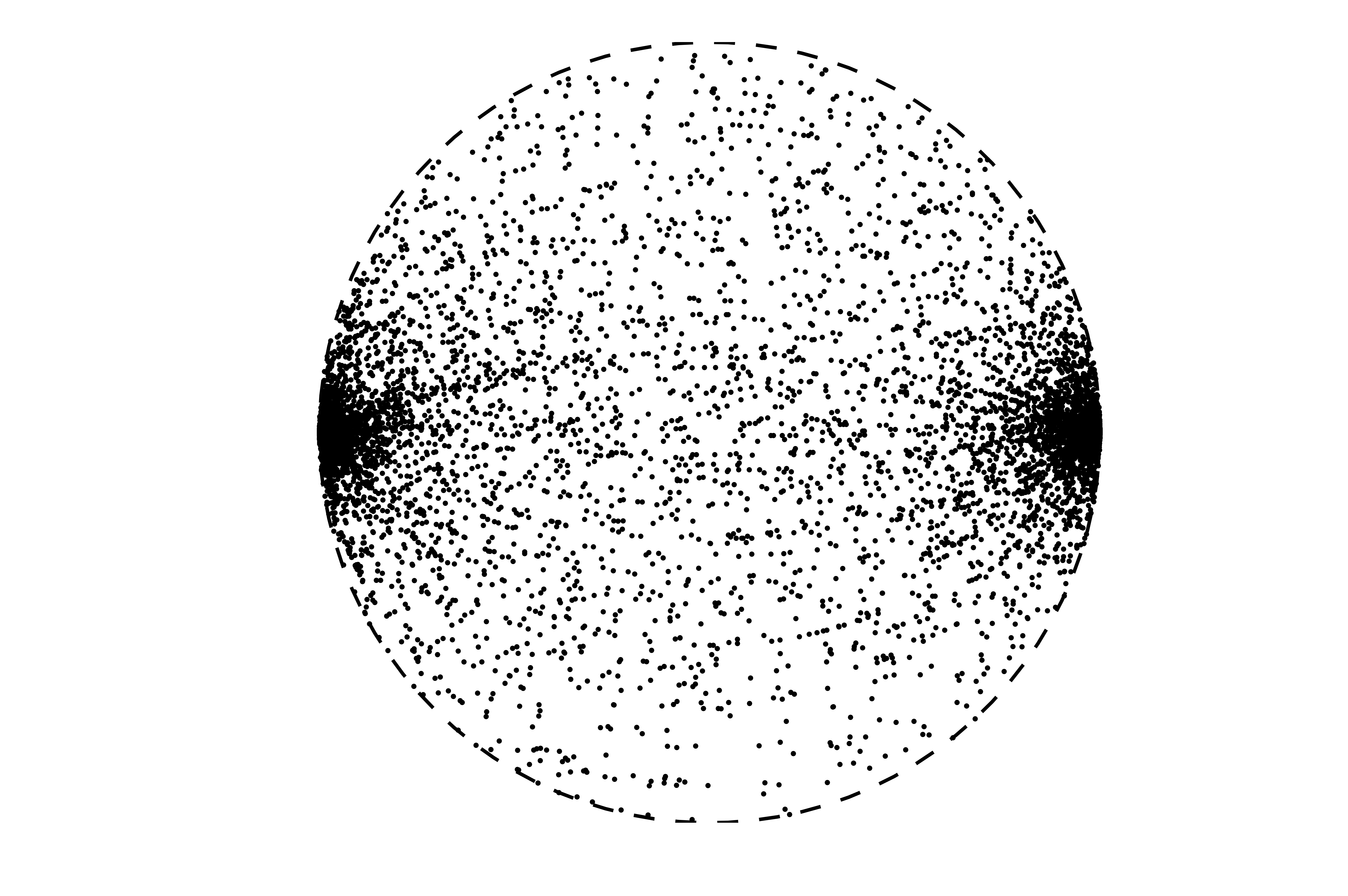} \\
(c) $p=0.9$
\end{tabular}
\caption{A realization of the random walk on Apollonian circles with $\alpha=1$ (300,000 steps). Here we have $x_n$ are i.i.d drawn from uniform distribution on $(-1,1),$ and $\varsigma_0$ is distributed uniformly over $\bigr[-\frac{\pi}{2},\frac{\pi}{2}\bigr].$}
\label{fig3}
\end{figure}

\begin{figure}
\centering
\begin{tabular}{cc}
\includegraphics[scale = 0.6, angle = 180]{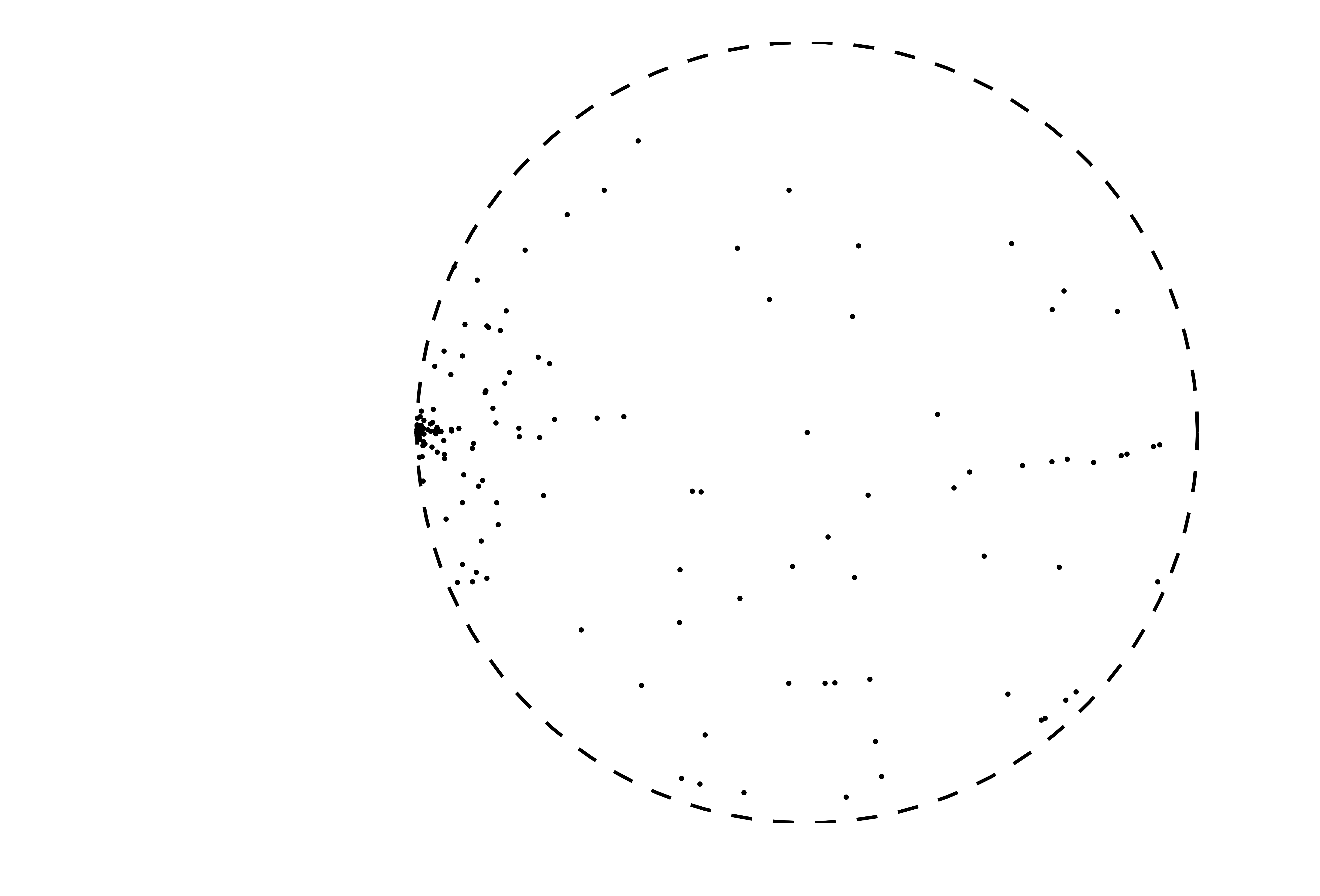} &
\includegraphics[scale = 0.6, angle = 180]{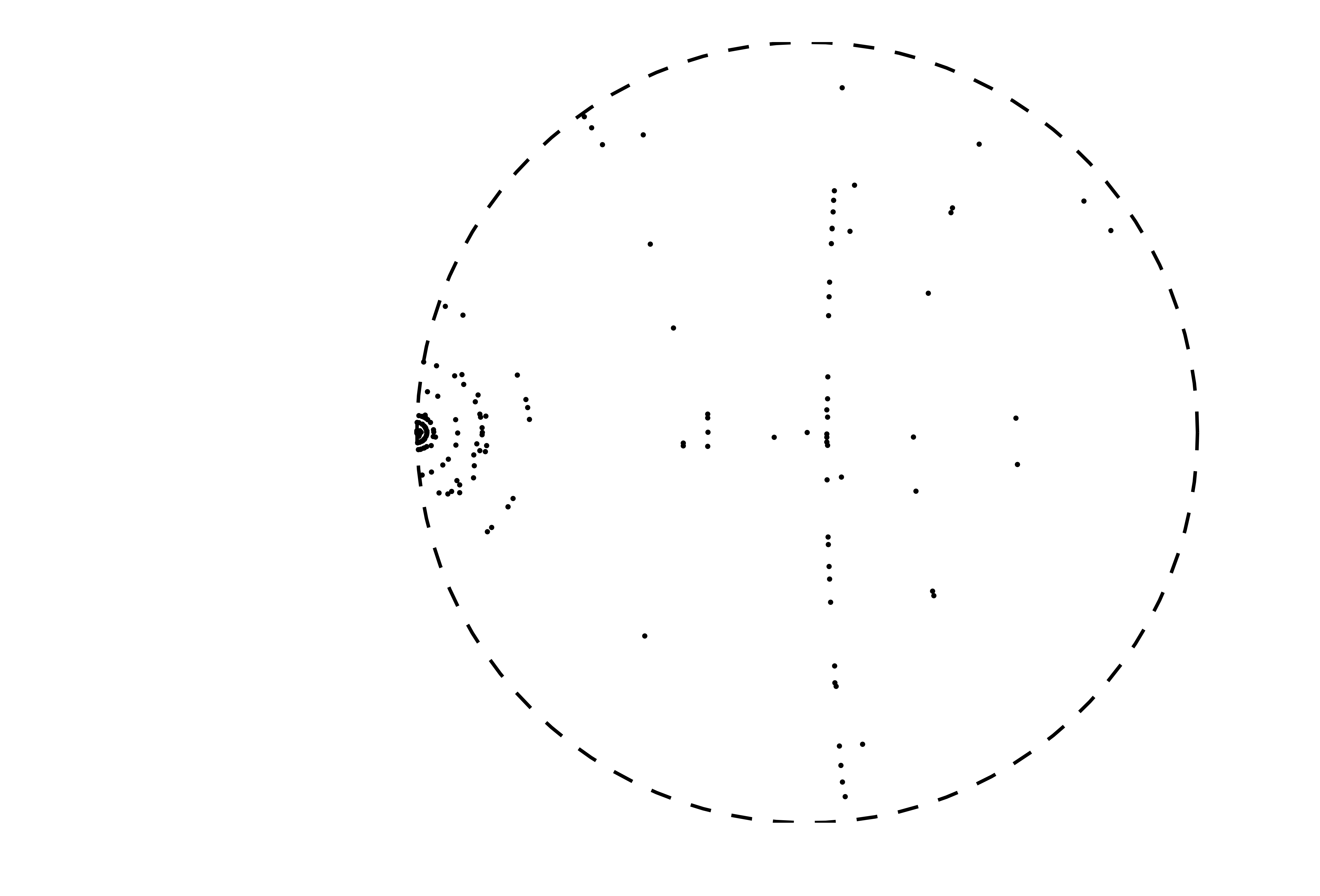} \\
(a) $p=0.5$ & (b) $p=0.1$
\end{tabular}
\begin{tabular}{c}
\includegraphics[scale = 0.6]{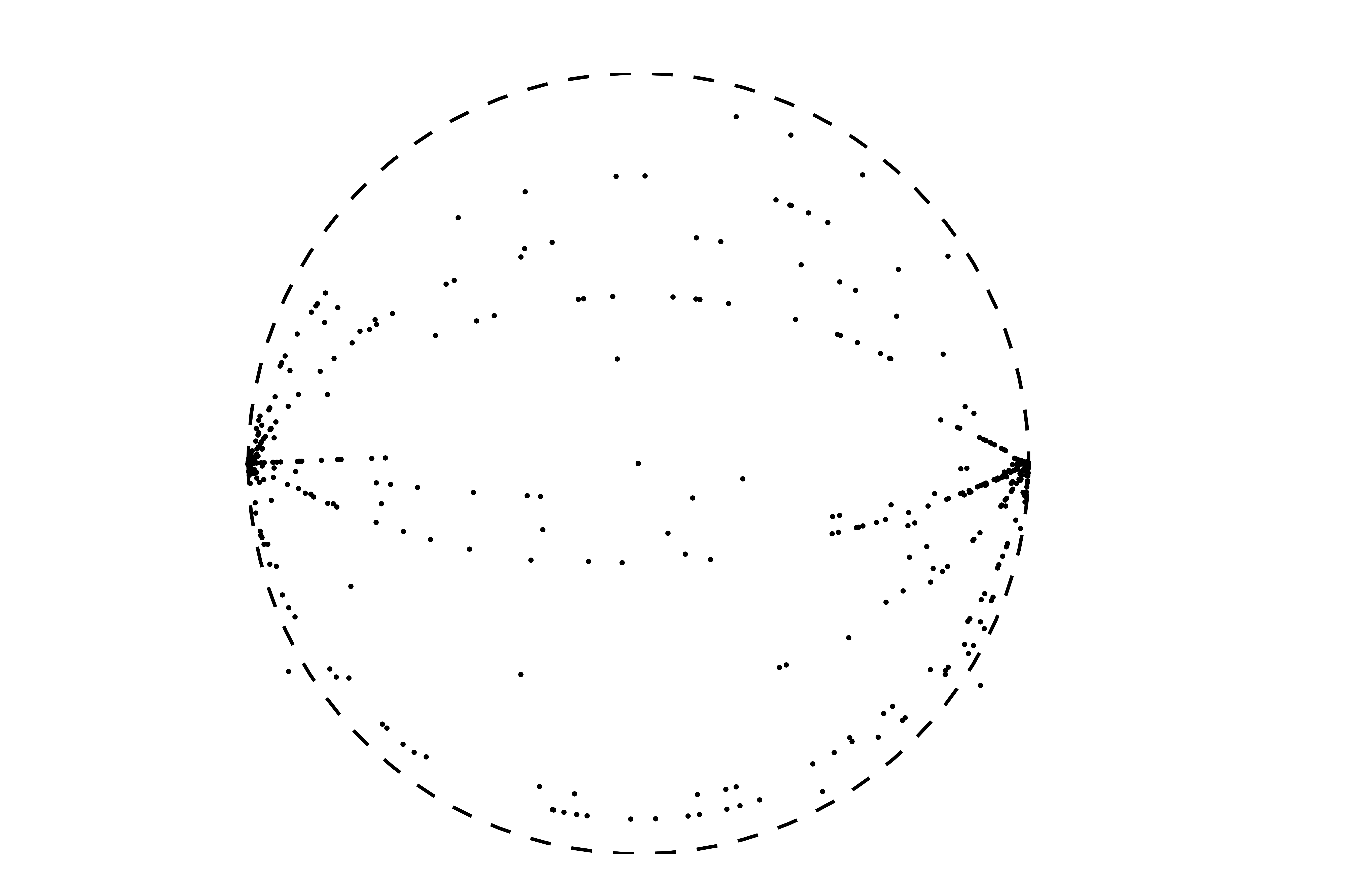} \\
(c) $p=0.9$
\end{tabular}
\caption{A realization of the random walk on Apollonian circles with $\alpha=1$ (300,000 steps). Here $x_n$ are i.i.d drawn from a triangular distribution on $(-1,1)$ with mode $0.1,$ and $\varsigma_0$ is distributed uniformly over $\bigr[-\frac{\pi}{2},\frac{\pi}{2}\bigr].$}
\label{fig4}
\end{figure}

\end{document}